\documentclass[12pt]{amsart}
\usepackage{amsmath}
\usepackage{amssymb}
\usepackage{amsfonts}
\usepackage{amsthm}
\usepackage{mathrsfs}
\usepackage[all]{xy}
\theoremstyle{definition}
\newtheorem*{theorem*}{Theorem}
\newtheorem{theorem}{Theorem}[section]
\newtheorem{lemma}[theorem]{Lemma}

\newtheorem{proposition}[theorem]{Proposition}

\newtheorem{conj}[theorem]{Conjecture}

\theoremstyle{definition}
\newtheorem{definition}[theorem]{Definition}
\newtheorem{remark}[theorem]{Remark}

\newtheorem{example}[theorem]{Example}

\begin{document}
\title{Stability conditions for Gelfand-Kirillov subquotients of category $\mathcal{O}$} 
\author{Vinoth Nandakumar}
\address{Dept. of Mathematics, University of Utah}
\email{vinoth.90@gmail.com}
\maketitle
\begin{center}\textit{In memory of Andrei Tarkovsky} \end{center}
\setcounter{tocdepth}{1}
\begin{abstract} Recently, Anno, Bezrukavnikov and Mirkovic have introduced the notion of a “real variation of stability conditions” (which is related to Bridgeland's stability conditions), and construct an example using categories of coherent sheaves on Springer fibers. Here we construct another example, by studying certain sub-quotients of category $\mathcal{O}$ with a fixed Gelfand-Kirillov dimension.  We use the braid group action on the derived category of category $\mathcal{O}$, and certain leading coefficient polynomials coming from translation functors. Consequently, we use this to explicitly describe a sub-manifold in the space of Bridgeland stability conditions on these sub-quotient categories, which is a covering space of a hyperplane complement in the dual Cartan. \end{abstract}
\tableofcontents

\section{Introduction}

In this paper, we will study stability conditions (and real variations thereof) for certain sub-quotients of category $\mathcal{O}$ for a semi-simple Lie algebra $\mathfrak{g}$. Building upon Bridgeland's work on stability conditions, in \cite{abm}, Anno, Bezrukavnikov and Mirkovic introduce the concept of a `real variations of stability conditions', give an example of representation theoretic significance (involving categories of coherent sheaves on Springer fibers), and study the space of stability conditions on these categories. Here, following their approach, we will give another example of a similar nature, by looking at Gelfand-Kirillov sub-quotients of category $\mathcal{O}$, and describe an explicit sub-manifold in the space of Bridgeland stability conditions for this category (which is homeomorphic to a covering space of a certain hyperplane complement in $\mathfrak{h}^*$). The main ingredients in our construction are certain leading coefficient polynomials defined for modules in category $\mathcal{O}$, and the braid group action on the derived categories of these sub-quotients. 

In \cite{bridgeland}, under certain technical assumptions, Bridgeland associates to each triangulated category $\mathcal{D}$ a complex manifold $\text{Stab}(\mathcal{D})$ parametrizing stability conditions on $\mathcal{D}$; further, the group $\text{Aut}(\mathcal{D})$ of exact auto-equivalences acts on $\text{Stab}(\mathcal{D})$ and preserves a certain natural distance function. Thus this gives a new topological invariant of triangulated categories; Bridgeland gives an explicit description of this manifold was given in \cite{bridgeland2} for categories of coherent sheaves on K3 surfaces, and in \cite{bridgeland3} for categories of coherent sheaves on resolutions of Kleinian singularities. The framework developed by Anno, Bezrukavnikov and Mirkovic in \cite{abm} is particularly amenable to the study of stability conditions on categories that appear in representation theory; in \cite{zhao}, Zhao uses these techniques to give another such example involving representation categories of rational Cherednik algebras in positive characteristic. In contrast to the constructions in \cite{abm} and \cite{zhao}, we will not use much geometry, or appeal to any localization-type results (e.g. Beilinson-Bernstein).

While the main results in \cite{abm} are phrased using categories of coherent sheaves on Springer fibers, they can be re-phrased using blocks of representations of Lie algebras in positive characteristic (using the equivalences in \cite{bmr}). Our main result may be viewed, loosely, as a characteristic $0$ analogue of Proposition $1$ and Theorem $1$ from \cite{abm}. The main difference is that, in their set-up, one may define the central charge map by simply applying a translation functor, and looking at the dimension of resulting module; in this case, however, we must use certain ``leading coefficients'' which are only defined on these Gelfand-Kirillov sub-quotients. Checking that our data gives a valid example of a real variation of stability conditions amounts to checking a number of compatibilities between the braid group action on these Gelfand-Kirillov questions, and these leading coefficient polynomials. 

The Gelfand-Kirillov filtration on category $\mathcal{O}$ is a source of rich connections with other areas of representation theory, including primitive ideals in the universal enveloping algebra, two-sided Kazhdan-Lusztig cells in the Weyl group, and nilpotent orbits. Given an irreducible representation $L_w$, (with highest weight $w \cdot 0$) in category $\mathcal{O}$, by taking its associated graded one obtains a module over $S(\mathfrak{g})$, whose support $\text{supp}(L_w)$ lies inside $\mathfrak{n}^+$; further, one can determine the Gelfand-Kirillov dimension of $L_w$ from the dimension of its support $\text{supp}(L_w)$. It is known that the $G$-saturation of $\text{supp}(L_w)$ is a nilpotent orbit, and two irreducibles $L_w$ and $L_y$ give rise to the same orbit iff they lie inside the same two-sided Kazhdan-Lusztig cell (this was proven by Barbasch, Joseph, Vogan et al; see \cite{bv}, \cite{vogan}, \cite{jos}, \cite{jos2} for more results in this direction). Thus we hope that the results which we prove about these Gelfand-Kirillov sub-quotients will also be of interest to those who study these aspects of representation theory.  

Let us now briefly summarize the contents of this paper in more detail. 

\subsection{Real variations of stability conditions} Inspired by Bridgeland's theory of stability conditions, in Section $1.4$ of \cite{abm}, Anno,  Bezrukavnikov and Mirkovic define the notion of a ``real variation of stability conditions" on a triangulated category $\mathcal{C}$. It consists of the following: a discrete collection of hyperplanes $\Sigma$ in an $\mathbb{R}$-vector space $V$, a polynomial map $Z: V \rightarrow (K^0(\mathcal{C}) \otimes \mathbb{R})^*$ (known as the ``central charge''), and a collection of t-structures on $\mathcal{C}$ indexed by connected components of $V \backslash \Sigma$ (known as ``alcoves''), satisfying some compatibilities (see Section $2.4$ below for more details).  

They then construct an example with $\mathcal{C} = D^b(\text{Coh}_{\mathcal{B}_e}(\widetilde{S}_e))$; here $e$ is a nilpotent, $\mathcal{B}_e$ is the corresponding Springer fiber, and $\widetilde{S}_e$ is the pre-image to the Slodowy slice at $e$ under the Springer resolution map. The hyperplane arrangement $\Sigma$ in question is the set of affine co-root hyperplanes in $V = \mathfrak{h}^*_{\mathbb{R}}$, and the central charge $Z: \mathfrak{h}^* \rightarrow (K^0(\mathcal{C}) \otimes \mathbb{R})^*$ is defined by the property that $Z(\lambda)[F]$ is the Euler characteristic of $\mathcal{F} \otimes \mathcal{O}(\lambda)$ given $\mathcal{F} \in \mathcal{C}$, $\lambda \in \Lambda$. The map $\tau$ from $\textbf{Alc}$ to the set of bounded $t$-structures on $D^b(\text{Coh}_{\mathcal{B}_e}(\widetilde{S}))$ constructed using the theory of ``exotic t-structures'', which were developed in greater generality by Bezrukavnikov and Mirkovic in \cite{bm} (see, for instance, Section $1.8.2$). As an application, in Section $4.2$, they explicitly describe a sub-manifold in the space of Bridgeland stability conditions for the category $\mathcal{C}$, which is homeomorphic to a covering space of a certain hyperplane complement in the Cartan. %For a detailed exposition of Bridgeland stability conditions, we refer the reader to Bridgeland's expository paper, \cite{bridgeland}; see also \cite{bridgeland2} and \cite{bridgeland3} for some related work. 

\subsection{Gelfand-Kirillov filtration, leading coefficient polynomials, and the braid group action} $\\$ 
So let $\mathfrak{g}$ be a semisimple Lie algebra over $\mathbb{C}$, with a fixed triangular decomposition $\mathfrak{g} = \mathfrak{n}^+ \oplus \mathfrak{h} \oplus \mathfrak{n}^-$. Then the BGG category $\mathcal{O}$ is the category of all finitely generated $U(\mathfrak{g})$-modules that are $\mathfrak{h}$-diagonalizable and locally $U(\mathfrak{n}_+)$-nilpotent. This category splits into blocks in accordance with the action of the center $Z(U(\mathfrak{g}))$ of the enveloping algebra; we will be interested primarily in the principal block $\mathcal{O}_0$. 

In Section $2.1$, we start by defining two quantities that one can associate two quantities to a given a module $M \in \mathcal{O}$: its Gelfand-Kirillov dimension $\text{GK}(M)$, the degree of a certain polynomial $p_M$, and $\underline{\text{LC}}(M)$, its leading coefficient. The polynomial $p_M$ is, roughly speaking, defined by the equality $p_M(i) = \text{dim}(M_i)$, where $M_i = (U(\mathfrak{g})^{\leq i}) v$; here $v$ is any vector generating $M$, and $\{ U(\mathfrak{g})^{\leq i} \}$ is the PBW filtration on $U\mathfrak{g}$. It will be more convenient for us to use the quantity $\text{LC}(M)$, which is a variant of $\underline{\text{LC}}(M)$ that is defined using a different filtration. Instead of working with $\mathcal{O}_0$ itself, we will be working the ``Gelfand-Kirillov sub-quotients'' $\mathcal{O}_{0}^d$, which consist of modules with GK-dimension $d$, modulo those with smaller dimension. 

In Section $2.2$, we will recall how the braid group $\mathbb{B}_W$ acts on the derived category $D^b(\mathcal{O}_0)$ (here the simple reflections act via wall-crossing functors). We also recall why the braid group action factors through to the quotient category $\mathcal{C} = D^b(\mathcal{O}_0^d)$. For each $w \in W$, one can look at a lift $\tilde{w} \in \mathbb{B}_W$; denote by $\tau(w)$ the image of the tautological t-structure on $D^b(\mathcal{O}_0^d)$ by the automorphism corresponding to $\tilde{w}$. These t-structures will be important to us.   

\subsection{The main construction.} $\\$ Now that the key ingredients are in place (ie. leading coefficient polynomials, Gelfand-Kirillov sub-quotients and the braid group action), we will use them to construct stability conditions. We summarize the main results below; see Section \ref{mainres} (and Section \ref{bridge}) for more details about the first (and second) parts respectively. We will deduce the second part from the first using techniques developed in \cite{abm}. 

\textbf{Main Theorem:} \begin{itemize} \item Let the hyperplane arrangement $\Sigma$ will be the set of linear co-root hyperplanes in $V =  \mathfrak{h}^*_{\mathbb{R}}$. Given $M \in \mathcal{O}^{\textbf{d}}_0$, the central charge map $Z: \mathfrak{h}_{\mathbb{R}}^* \rightarrow (K^0(\mathcal{C}) \otimes \mathbb{R})^*$ is defined by uniquely determined by the property that $Z(\lambda)[M] = \text{LC}(T_{0 \rightarrow \lambda} M)$ for $\lambda \in \Lambda^+$ (here $T_{0 \rightarrow \lambda}$ denotes the translation functor). The set of alcoves is in bijection with $W$; for the alcove indexed by $w \in W$, we associate to it the t-structure $\tau(w)$. Then this datum constitutes a real variation of stability conditions. \item $\widetilde{V_{\text{reg}}}$ is defined to be a covering space of a certain hyperplane complement in $\mathfrak{h}^*_{\mathbb{C}}$, and $\pi: \widetilde{V_{\text{reg}}} \rightarrow \mathfrak{h}^*_{\mathbb{C}}$ the projection map. The map $\textbf{Z}: \mathfrak{h}^*_{\mathbb{C}} \rightarrow (K^0(\mathcal{C}) \otimes \mathbb{C})^*$ is defined to be a complexification of the map $Z$. Then there exists a unique map (of manifolds) $\imath$, from $\widetilde{V_{\text{reg}}}$ to the space $\text{Stab}(\mathcal{C})$ of locally finite Bridgeland stability conditions on $\mathcal{C}$, such that: \begin{enumerate} \item We have the commuting square $$\overline{\pi} \circ \imath= \sqrt{-1} \textbf{Z} \circ \pi$$ where $\overline{\pi}: \text{Stab}(\mathcal{C}) \rightarrow K^0(\mathcal{C})^*$ is the natural projection map. \item For a point $z$ lying in the alcove corresponding to $w \in W$, the underlying t-structure of the stability condition $\imath(z)$ is $\tau(\underline{w})$. Further, the map $\imath$ is compatible with the action of $\mathbb{B}_W$, which acts on the source by deck transformations, and on the target via the action on $\mathcal{C}$. \end{enumerate} \end{itemize} 

\subsection{Proof of main results:} Before embarking on the proof, in Section $3$ we prove some results about the leading coefficient functions which we will need. Given $M \in \mathcal{O}_0^d$, the central charge map above is defined by stipulating that $Z(\lambda)[M] = LC(T_{0 \rightarrow \lambda} M)$ when $\lambda$ is dominant; here we will show that this map is well-defined (ie. that such a polynomial exists). In order to do this, we give a way of extracting the leading coefficient $\text{LC}(M)$ from the character of $M$ (or more precisely, a certain Taylor expansion of the character). As a result, we are also able to show a certain compatibility relation between the braid group action and these leading coefficient polynomials. 

In Section \ref{3}, first we prove the first part of Theorem 1: that the two axioms defining a real variation of stability conditions are satisfied. The first condition almost follows from the results of Section \ref{2}; however, we need some additional machinery to show that $\langle Z(x), [M] \rangle$ is strictly positive (and not just non-negative). To show the second condition, we examine how translation to the wall interacts with these leading coefficient polynomials; we appeal to the theory of harmonic polynomials to show the filtrations appearing in the second axiom (relating to zeros that these functions have on co-root hyperplanes) are two-step filtrations. The second part of the theorem is proven in Section \ref{bridge}, using the techniques developed in Section $4.2$ of \cite{abm}. 

We close by describing some avenues for further work, and some open questions.

\subsection{Acknowledgements} I would like to thank my advisor, Roman Bezrukavnikov, for suggesting this problem, and for numerous helpful discussions. I am also much indebted to David Vogan for some crucial insights relating to Proposition \ref{taylor}; to Gufang Zhao and David Yang for some helpful conversations; and to Catharina Stroppel for helpful comments. 

\section{Statement of main results.} \label{2}

\subsection{Gelfand-Kirillov dimension and leading coefficients.} \label{2.1}

Let $\mathcal{O}_0$ be the principal block of category $\mathcal{O}$. Given a module $M \in \mathcal{O}$, recall that its Gelfand-Kirillov dimension is defined as follows. Consider the PBW filtration on $U(\mathfrak{g})$, where $U(\mathfrak{g})^{\leq i}$ denotes the subspace of $U(\mathfrak{g})$ spanned by products $x_1 x_2 \cdots x_k$ with $k \leq i$ and $x_j \in \mathfrak{g}$. 

Let $\underline{M}_0 \subseteq M$ be a vector sub-space which generates $M$ as a $U(\mathfrak{g})$-module, and let $\underline{M}_i = U(\mathfrak{g})^{\leq i} \cdot \underline{M}_0$. Let us start be collecting some well-known facts (Proposition \ref{a1}, Lemma \ref{a2} and Lemma \ref{gkdim}) about this filtration. While these results certainly aren't new, we include proofs for the reader's convenience; see Mazorchuk and Stroppel's paper \cite{ms} for a more detailed exposition. We expect that some of the other results in this sub-section may also be known to experts, but we were unable to find them in the literature. 

\begin{proposition} \label{a1} There exists a polynomial $p$ such that: for all $i$ sufficiently large, $\text{dim}(\underline{M}_i) = p(i)$. The leading term of this polynomial $p$ does not depend on the choice of subspace $M_0$. \end{proposition} \begin{proof} With this grading, the associated graded of $U(\mathfrak{g})$ is $S(\mathfrak{g})$ (where all elements of $\mathfrak{g}$ have degree $1$); let $\underline{M} = \text{gr}(M)$. Then by the theory of Hilbert polynomials, there exists a polynomial $q$ such that: $$ \text{dim}(\underline{M}_i) t^i = \frac{P(t)}{(1-t)^{\text{dim}(\mathfrak{g})}} $$ It follows that there exists a polynomial $p$, such that $\text{dim}(\underline{M}_i) = p(i)$ for $i$ sufficiently large. Suppose now that we pick a different subspace $\underline{M}_0'$, which gives rise to a filtration $\underline{M}'_i$ with dimension polynomial $p'$. Then $\underline{M}_0' \subseteq \underline{M}_k$ for some $k$, so $\underline{M}_i' \subseteq \underline{M}_{i+k}$, and $p'(i) \leq p(i+k)$ for $i$ large. Similarly, for some $l$, $p(i) \leq p'(i + l)$, provided that $i$ large. These two inequalities imply that $p$ and $p'$ have the same leading term, i.e. that the leading term doesn't depend on the choice of subspace $\underline{M}_0$. \end{proof}

\begin{definition} The degree of the polynomial $p$, $\text{GK}(M)$, is known as the Gelfand-Kirillov dimension of $M$; denote the leading coefficient of $p$ by $\underline{\text{LC}}(M)$. \end{definition}

From the following self-evident Lemma, we deduce that the set of all $M \in \mathcal{O}_0$ with Gelfand-Kirillov dimension at most $d$ (for some $d \in \mathbb{Z}_{\geq 0}$) forms a Serre sub-category.

\begin{lemma} \label{a2} Given an exact sequence $0 \rightarrow A \rightarrow B \rightarrow C \rightarrow 0$, then $$\text{max} \{ \text{GK}(A), \text{GK}(C) \} = \text{GK}(B)$$ \end{lemma}

\begin{definition} Denote by $\mathcal{O}_0^{\leq d}$ (resp., $\mathcal{O}_0^{< d}$) be the Serre sub-category of $\mathcal{O}_0$ consisting of objects $M$ with Gelfand-Kirillov dimension at most $d$ (resp., strictly less than $d$). Let $\mathcal{O}_0^{d}$ denote the Serre quotient category $\mathcal{O}_0^{\leq d} / \mathcal{O}_0^{< d}$. \end{definition}

\begin{lemma} \label{gkdim} The Verma module $\Delta(\lambda) = U(\mathfrak{g}) \otimes_{U(\mathfrak{b})} \mathbb{C}_{\lambda}$ has Gelfand-Kirillov dimension $|\Delta^+|$. More generally, given a parabolic sub-algebra  $\mathfrak{p} \supseteq \mathfrak{b}$ and a finite-dimensional irreducible representation $V_{\lambda}$ of $\mathfrak{p}$ (which factors through to the Levi sub-algebra $\mathfrak{l}$), then the parabolic Verma module $\Delta_{\mathfrak{p}}(\lambda) = U(\mathfrak{g}) \otimes_{U(\mathfrak{p})} {V}_{\lambda}$ has Gelfand-Kirillov dimension $|\Delta^+| - N(\mathfrak{p})$, where $N(\mathfrak{p}) = |\alpha \in \Delta^+: F_{\alpha} \in \mathfrak{p}|$. 
\end{lemma} \begin{proof} Pick a basis $\{ v_1, \cdots, v_k \}$ for $V_{\lambda}$, where $k = \text{dim}(V_{\lambda})$. Using the PBW theorem, the parabolic Verma has basis $$\prod_{\alpha \in \Delta^+_{\mathfrak{p}}, 1 \leq i \leq k} F_{\alpha}^{n_{\alpha}} v_i $$ Here we have fixed an order on the set $\Delta^+_{\mathfrak{p}} = \{ \alpha \in \Delta^+: F_{\alpha} \in \mathfrak{p} \}$; and the $n_{\alpha}$-s are arbitrary positive integers. Picking $\underline{M}_0 = V_{\lambda}$, the above product lies in $\underline{M}_n$, where $n = \sum_{\alpha \in \Delta^+} n_{\alpha}$. Thus: $$ \text{dim }(\underline{M}_n) = \text{dim }(V_{\lambda}) |  \{ n_{\alpha} \} : \sum_{\alpha \in \Delta^+_{\mathfrak{p}}} n_{\alpha} \leq n | $$ The result now follows.  \end{proof}

\begin{example} \label{ex1} Let us consider the example with $\mathfrak{g} = \mathfrak{sl}_3$, and calculate the Gelfand-Kirillov dimensions of the simple objects. 

The simple objects in $\mathcal{O}_0$ are $L(w \cdot 0)$ with $w \in S_3$. When $w = 1$, $L(w \cdot 0)$ is the trivial $1$-dimensional module, and clearly has Gelfand-Kirillov dimension $0$. When $w = w_0 = s_1 s_2 s_1$, $L(w \cdot 0) = \Delta(w \cdot 0)$ since the Verma module is irreducible; and has Gelfand-Kirillov dimension $3$ using the above Lemma.  

Let $\mathfrak{p}_1$ (resp. $\mathfrak{p}_2$) be the parabolic sub-algebra containing $F_{\alpha}$, where $\alpha = \epsilon_1 - \epsilon_2$ (resp. where $\alpha = \epsilon_2 - \epsilon_3$). Consider the corresponding parabolic sub-categories $\mathcal{O}^{\mathfrak{p}_1}_0, \mathcal{O}^{\mathfrak{p}_2}_0 \subset \mathcal{O}_0$. From Chapter $9$ of \cite{h}, we know that $L(w \cdot 0) \in \mathcal{O}^{\mathfrak{p}_I}_0$ (where $\mathfrak{p}_I$ is the parabolic subalgebra corresponding to a subset $I$ of the set of simple roots) precisely when $\langle w \cdot 0, \check{\alpha} \rangle > 0$ for all $\alpha \in I$. Thus $L( (s_2 s_1) \cdot 0 ), L(s_2 \cdot 0) \in \mathcal{O}^{\mathfrak{p}_1}_0$, and $L( (s_1 s_2) \cdot 0 ), L(s_1 \cdot 0) \in \mathcal{O}^{\mathfrak{p}_2}_0$.

 In fact, one can prove that the parabolic Verma modules $\Delta_{\mathfrak{p}_1}((s_2 s_1) \cdot 0)$ and $\Delta_{\mathfrak{p}_2}((s_1 s_2) \cdot 0)$ are irreducible (either directly, or by using the criterion in Section $9.12$ of \cite{h}); thus $L((s_2 s_1) \cdot 0) = \Delta_{\mathfrak{p}_1}((s_2 s_1) \cdot 0)$ and $L((s_1 s_2) \cdot 0) = \Delta_{\mathfrak{p}_2}((s_1 s_2) \cdot 0)$. Further, one can prove that we have exact sequences: \begin{align*} 0 \rightarrow \Delta_{\mathfrak{p}_1}((s_2 s_1) \cdot 0) &\rightarrow \Delta_{\mathfrak{p}_1}(s_2 \cdot 0) \rightarrow L(s_2 \cdot 0) \rightarrow 0 \\ 0 \rightarrow \Delta_{\mathfrak{p}_2}((s_1 s_2) \cdot 0) &\rightarrow \Delta_{\mathfrak{p}_2}(s_1 \cdot 0) \rightarrow L(s_1 \cdot 0) \rightarrow 0 \end{align*} Using the above Lemma, it is now easy to check that the four simples $L((s_1 s_2) \cdot 0), L((s_2 s_1) \cdot 0)$, $L(s_1 \cdot 0), L(s_1 \cdot 0)$ have Gelfand-Kirillov dimension $2$. $\blacksquare$ \end{example}

For our purposes it will be more convenient to modify the definition of leading coefficients. To this end, note that $M = U(\mathfrak{n}_{-}) M_0$, and define a different grading on $U(\mathfrak{n}_{-})$ by setting $\text{deg}(F_{\alpha}) = \langle \check{\rho}, \alpha \rangle$ (so, in particular, $\text{deg}(F_{\alpha})=1$ when $\alpha$ is a simple root). This gives a filtration on $U(\mathfrak{n}_{-})$, where $$ U(\mathfrak{n}_{-})^{\leq i} = \{ \text{span} (\prod_{\alpha \in \Delta^+} F_{\alpha}^{n_{\alpha}}) \; | \; \sum_{\alpha \in \Delta^+} n_{\alpha} \text{deg}(\alpha) \leq i \} $$
Define $M_i = U(\mathfrak{n}_{-})^{\leq i} \cdot M_0$. We will see in the next example that $\text{dim }(M_i)$ is no longer a polynomial in $i$; however, we will prove that the weaker statement below does hold.

\begin{definition} We say that a function $q: \mathbb{Z} \rightarrow \mathbb{Z}$ is ``quasi-polynomial", if there exists an integer $k$, and polynomials $q_0, q_1, \cdots, q_{k-1}$ with the same degree and leading coefficient, such that $q(n) = q_i(n)$ if $n \equiv i \; (\text{mod} \; k)$. \end{definition}

\begin{proposition} \label{quasipoly} There exists a quasi-polynomial function $q$, such that for $i$ sufficiently large, $$\text{dim} (M_i/M_{i-1}) = q(i)$$ Further, $\text{deg}(q) = \text{GK}(M)$, and the leading coefficient of $q$ does not depend on the choice of $M_0$. \end{proposition} \begin{proof} It is clear that the associated graded algebra of $U(\mathfrak{n}_{-})$ with respect to the grading described above, is $S(\mathfrak{n}_{-})$ (where the corresponding elements have the same grading); let $\overline{M} = \text{gr}(M)$. From the general theory of Hilbert polynomials, we deduce there exists a polynomial $q$ with the following property (more generally, this statement is true with $\langle \alpha,  \check{\rho} \rangle$ being replaced by the degrees of the generators).
\begin{align*} \sum_{i \geq 0} \text{dim}(M_i) t^i = \frac{P(t)}{\prod_{\alpha \in \Delta^+} (1 - t^{\langle \alpha,  \check{\rho} \rangle})} \end{align*}
Using the above formula, and inducting on the number of generators, it follows that there exist an integer $k$, and polynomials $q_0, q_1, \cdots, q_{k-1}$ such that $\text{dim}(M_n) = q_i(n)$ if $n \equiv i \; (\text{mod }k)$ for $n$ large. It remains to prove that the these polynomials have the same degree and leading coefficient.

Using primary decomposition for modules, we obtain a filtration $0 = \overline{M}_0 \subset \overline{M}_1 \subset \cdots \subset \overline{M}_{k-1} \subset \overline{M}_k = \overline{M}$, such that for each $i$, $\text{Ann}(\overline{M}_j/\overline{M}_{j-1}) = \mathfrak{p}_i$ for some prime ideal $\mathfrak{p}_i$; and $R/\mathfrak{p}_i$ acts injectively. It is sufficient to prove the above statement for each of the sub-quotients $\overline{M}_j/\overline{M}_{j-1}$. So we may assume that $\text{Ann}(\overline{M}) = \mathfrak{p}$ for some prime ideal $\mathfrak{p}$, and $R/\mathfrak{p}$ acts injectively on $\overline{M}$. The support of $M$ is contained in $(\mathfrak{n}_{-})^*$; we will identify $(\mathfrak{n}_{-})^*$ with $\mathfrak{n}_{+}$ via the Killing form.

It follows using Lemma \ref{support} below, that there exists an element $t$ of degree $1$, such that $t$ acts injectively on $M$. This means that we have injective map from $M_i$ to $M_{i+1}$, given by multiplication by $t$; thus $\text{dim}(M_i) \leq \text{dim}(M_{i+1})$, and $q_i(n) \leq q_{i+1}(n+1) \leq q_i(n+k)$ if $0 \leq i \leq k-2, n \equiv i \; (\text{mod }k)$. This implies that the polynomials $q_i$ and $q_{i+1}$ have the same degree and leading coefficient, as required. \end{proof}

\begin{lemma} \label{support} The support of the module  $\overline{M}$ is not contained inside the subvariety $[\mathfrak{n}^+, \mathfrak{n}^+]$ of $\mathfrak{n}^+$. \end{lemma} \begin{proof} Recall that an orbital variety is an irreducible component of the intersection $\mathbb{O} \cap \mathfrak{n}^+$, where $\mathbb{O}$ is a nilpotent orbit. Since $M$ is a module in category $\mathcal{O}$, it is well-known that the support of $\overline{M}$, is a union of orbital varieties (see Joseph, \cite{jos}, and Borho-Brylinski, \cite{bbry}, for a proof). We will show that no orbital variety is contained inside $[\mathfrak{n}^+, \mathfrak{n}^+]$. 

Pick $e \in \mathbb{O}$, and let $\mathcal{B}_e$ be the Springer fiber: $$ \mathcal{B}_e = \{ \mathfrak{b} \in \mathcal{B} \; | \; e \in \mathfrak{b} \} $$ In Claim $6.5.8$ of \cite{cg}, it is proven that the irreducible components of $\mathbb{O} \cap \mathfrak{n}^+$ are in bijection with $C^{\circ}(e)$-orbits on $\mathcal{B}_e$ (here $C^{\circ}(e)$ denotes the set of connected components of the centralizer of $e$). Suppose an orbital variety $Y$ is contained in $[\mathfrak{n}^+, \mathfrak{n}^+]$, and pick one of the corresponding components $X$ of $\mathcal{B}_e$. Define: $$ \mathcal{B}_e^{\circ} = \{ \mathfrak{b} \in \mathcal{B} \; | \; e \in [\mathfrak{n}^+, \mathfrak{n}^+] \} $$ Since $Y \subseteq [\mathfrak{n}^+, \mathfrak{n}^+]$, from the bijection sketched in Claim $6.5.8$ it is clear that $X \subseteq \mathcal{B}_e^{\circ}$. For each positive root $\alpha \in \Delta^+$, let $P_{\alpha} \supseteq B$ be the corresponding minimal parabolic and let $\pi_{\alpha}: G/B \rightarrow G/P_{\alpha}$ be the natural projection map. Since $X \subseteq \mathcal{B}_e^{\circ}$, it is easy to see that $X \subseteq \pi_{\alpha}^{-1}(\pi_{\alpha}(X)) \subseteq \mathcal{B}_e$. Since $\pi_{\alpha}^{-1}(\pi_{\alpha}(X))$ is irreducible, and $X$ is one of the irreducible components of $\mathcal{B}_e$, in fact we have that $X = \pi_{\alpha}^{-1}(\pi_{\alpha}(X))$ for each $\alpha$. 

Define an equivalence relation on points in $\mathcal{B}$ as the transitive closure of the following relation: if $x, y \in \mathcal{B}$, define $x \sim y$ if $\pi_{\alpha}(x) = \pi_{\alpha}(y)$ for some simple root $\alpha$. It is well-known, that in fact $x \sim y$ for any two points $x,y \in \mathcal{B}$ (see Spaltenstein's paper \cite{spalt} for a reference). Now pick any point $x \in X$; since $X = \pi_{\alpha}^{-1}(\pi_{\alpha}(X))$, any other point in the same equivalence class as $x$ is also in $X$; it follows that $X = \mathcal{B}$. This is only possible when $e=0$, and in this case one easily checks that the orbital variety $Y$ is not contained in $[\mathfrak{n}^+, \mathfrak{n}^+]$. Thus we have reached a contradiction, and so no orbital variety is contained in $[\mathfrak{n}^+, \mathfrak{n}^+]$. \end{proof}

\begin{remark} Above we have used some non-trivial facts about the support of modules in category $\mathcal{O}$; it is possible that the above proof can be simplified, and that the statement holds in greater generality. \end{remark}

\begin{definition} Define $\text{LC}(M)$ to be the leading coefficient of the quasi-polynomial $q$ from Proposition \ref{quasipoly}. \end{definition}

In fact, we conjecture that the two quantities $\text{LC}(M)$ and $\underline{\text{LC}}(M)$ differ by a constant:

\begin{conj} \label{conj} There exists a constant $C$, depending only on the Lie algebra $\mathfrak{g}$ and $d$, such that for all $M \in \mathcal{O}_\lambda^d$, we have: $$\underline{\text{LC}}(M) = C \cdot \text{LC}(M)$$ \end{conj}

\begin{example} \label{ex2} Returning to Example \ref{ex1}, let us compute $\text{LC}(M)$ and $\underline{\text{LC}}(M)$ when $M$ is a simple modules $L(w \cdot \lambda)$ lying in $\mathcal{O}_{\lambda}^1$ for $\mathfrak{g}=\mathfrak{sl}_3$. %When $M = L(\lambda)$ is the irreducible module with highest weight $\lambda$, clearly $\text{LC}(M) = \text{dim}(L(\lambda))$. 
First consider $M=L((s_2 s_1) \cdot \lambda)$ and recall that: $$L((s_2 s_1) \cdot \lambda) \simeq U(\mathfrak{g}) \otimes_{U(\mathfrak{\mathfrak{p}_1})} V_{(s_2 s_1) \cdot \lambda}$$ Pick a weight basis of $V_{(s_2 s_1) \cdot \lambda}$: $v_1, \cdots, v_k$ where $k = \text{dim}(V_{(s_2 s_1) \cdot \lambda})$. Then a basis for $L((s_2 s_1) \cdot \lambda)$ is given by $E_{31}^i E_{32}^j v_l$ where $i, j \geq 0, 1 \leq k \leq l$; and a basis for $M_{n}/M_{n-1}$ is given by $E_{31}^i E_{32}^j v_l$ where $2i + j = n$. Thus: \begin{align*} \text{dim}(M_n/M_{n-1}) &= | \{ (i, j) \; | \; 2i + j = n \} | \\ &= \begin{cases} k (\frac{n}{2} + 1) &\mbox{if } n \equiv 0 \\ k (\frac{n+1}{2}) & \mbox{if } n \equiv 1 \end{cases} \pmod{2} \\ \text{LC}(M) &= \frac{1}{2} \text{dim}(V_{(s_2 s_1) \cdot \lambda}) = \frac{1}{2} \langle \lambda + \rho, \check{\alpha_2} \rangle \end{align*} Similarly, one may compute that $\underline{\text{LC}}(M) = \langle \lambda + \rho, \check{\alpha_2} \rangle$. Using the descriptions of $L((s_1 s_2) \cdot \lambda), L(s_2 \cdot \lambda)$ and $L(s_1 \cdot \lambda)$ given in Example \ref{ex1}, we may also deduce that:  \begin{align*} \text{LC}(L((s_1 s_2) \cdot \lambda)) &= \frac{1}{2} \langle \lambda + \rho, \check{\alpha_1} \rangle, \; \underline{\text{LC}} (L((s_1 s_2) \cdot \lambda)) = \frac{1}{2} \langle \lambda + \rho, \check{\alpha_1} \rangle \\ \text{LC}(L(s_1 \cdot \lambda)) &= \frac{1}{2} \langle \lambda + \rho, \check{\alpha_2} \rangle, \; \underline{\text{LC}} (L(s_1 \cdot \lambda)) = \frac{1}{2} \langle \lambda + \rho, \check{\alpha_2} \rangle \\ \text{LC}(L(s_2 \cdot \lambda)) &= \frac{1}{2} \langle \lambda + \rho, \check{\alpha_1} \rangle, \; \underline{\text{LC}} (L(s_2 \cdot \lambda)) = \frac{1}{2} \langle \lambda + \rho, \check{\alpha_1} \rangle \qquad \qquad \blacksquare \end{align*} \end{example}

\subsection{Braid group action on derived category of $\mathcal{O}_0$.}

Let $\mathbb{B}_W$ denote the braid group associated to the Weyl group $W$. Here we will recall some facts about the action of $\mathbb{B}_W$ on the derived category $D^b(\mathcal{O}_0)$. None of the results in this section are new, but we sketch the proofs for the reader's convenience. 

Given a simple root $\alpha \in \Delta^+$, and $M \in D^b(\mathcal{O}_0)$, we will define the action of $\tilde{s}_{\alpha}$ on $M$ as follows. 

Recall that the wall-crossing functor $R_{\alpha}: \mathcal{O}_0 \rightarrow \mathcal{O}_0$ is defined as follows. First pick $\mu$ so that: $$ \langle \mu + \rho, \check{\alpha} \rangle = 0, \langle \mu + \rho, \check{\beta} \rangle > 0 \text{ if } \beta \in \Delta^+, \beta \neq \alpha$$

Then define $R_{\alpha} = T_{\mu \rightarrow 0} T_{0 \rightarrow \mu}$. It can be shown that $R_{\alpha}$ does not depend on the choice of $\mu$. Now define:  $$\Phi(\tilde{s}_{\alpha}) M = \text{Cone}(M \rightarrow R_{\alpha} M)$$

The following theorem is originally due to Beilinson-Bernstein:

\begin{theorem} The above action gives rise to an action of the braid group $\mathbb{B}_W$ on $D^b(\mathcal{O}_0)$. \end{theorem}
\begin{proof} See Theorem $5.7$, Proposition $5.3$ and Lemma $5.10$ in Mazorchuk-Stroppel's paper \cite{ms2}; see also the proof of Corollary $9.6$ in the expository article \cite{bk} for a different approach. \end{proof}

However, since we are dealing the sub-quotients $\mathcal{O}_0^d$, we would like to have an action of the braid group $\mathbb{B}_W$ on $D^b(\mathcal{O}_0^d)$. Given $M \in D^b(\mathcal{O}_0)$, define $$\text{GK}(M) = \text{max}_{i \in \mathbb{Z}} \text{GK}(H^i(M))$$

\begin{proposition} Given $w \in \mathbb{B}_W$, then $\text{GK}(M) = \text{GK}(\Phi(w) \cdot M)$. In particular, the action of $\mathbb{B}_W$ on $D^b(\mathcal{O}_0)$ induces an action of $\mathbb{B}_W$ on $D^b(\mathcal{O}_0^d)$. \end{proposition} \begin{proof} Since the simple reflections $\tilde{s}_{\alpha}$ and their inverses generate the braid group, it is sufficient to show that $M$ and $\Phi(\tilde{s}_{\alpha}) M$ have the same Gelfand-Kirillov dimension. First, $T_{\mu \rightarrow 0} T_{0 \rightarrow \mu} M$ has Gelfand-Kirillov dimension at most equal to that of $M$ (since tensoring by a finite-dimensional does not increase the Gelfand-Kirillov dimension). Therefore: $$ \text{GK}(\Phi(\tilde{s}_{\alpha}) M) \leq \text{GK}(M)$$ Since $\Phi(\tilde{s}_{\alpha}^{-1}) M = \text{Cone}(T_{\mu \rightarrow 0} T_{0 \rightarrow \mu}(M) \rightarrow M)[-1]$, by a similar argument the reverse inequality also holds, and the conclusion follows. In fact, the stronger statement that the braid group action preserves the support of a module is true (i.e. $\text{supp}(M) = \text{supp}(\Phi(w) \cdot M)$); see Joseph (\cite{jos2}) for a more detailed discussion. %What about when $M$ is a complex?
 \end{proof}

The induced action of the braid group on the Grothendieck group $K^0(\mathcal{O}_0)$ is particularly simple to describe; we record it for later use (see \cite{bk} for a proof). 

\begin{lemma} \label{WGroth} In the Grothendieck group of $\mathcal{O}_0$, we have the equality $[\Phi(\tilde{s}_{\alpha}) \Delta_{w \cdot 0}] = [ \Delta_{(w s_{\alpha}) \cdot 0} ]$. \end{lemma}

\subsection{The central charge map}

Given $M \in \mathcal{O}_0^d$, and $\lambda \in \Lambda^+$ dominant, we have $T_{0 \rightarrow \lambda} M \in \mathcal{O}_{\lambda}^d$. 

\begin{proposition} \label{Z} There exists a unique polynomial function $Z: \mathfrak{h}^* \rightarrow (K^0(\mathcal{C}) \otimes \mathbb{R})^*$ such that: \begin{itemize} \item $Z(\lambda)([M]) = \text{LC}(T_{0 \rightarrow \lambda} M)$ for $\lambda \in \Lambda^+$, $M \in \mathcal{O}_0^d$. \item $ Z(y^{-1} \cdot \lambda) ([\Phi(\tilde{y}) M]) = Z(\lambda)([M])$ for $\lambda \in \mathfrak{h}^*$, $M \in D^b(\mathcal{O}_0^d)$, $y \in W$. \end{itemize} \end{proposition}

\begin{example} \label{ex3} Before proving this proposition, let us return to the example with $\mathfrak{g} = \mathfrak{sl}_3$, $d=1$ and verify the first part of the above Proposition by calculating the function $Z$. More precisely, we will calculate $Z(\lambda)[M]$ when $M$ is one of the four simple objects in $\mathcal{O}_0^{d}$.

First let us calculate $Z(\lambda)[M_1]$ for $M_1 = L((s_2 s_1) \cdot 0) = \Delta_{\mathfrak{p}_1}((s_2 s_1) \cdot 0)$. In this case, using the computations from Example \ref{ex2}, we have: \begin{align*} T_{0 \rightarrow \lambda} M_1 \simeq L((s_2 s_1) \cdot \lambda) &\simeq \Delta_{\mathfrak{p}_1}((s_2 s_1) \cdot \lambda) \\ Z(\lambda)[M_1] = \frac{1}{2} \langle \lambda + \rho, \check{\alpha_2} \rangle \end{align*} Similarly, we compute that: \begin{align*} Z(\lambda)[ L((s_1 s_2) \cdot 0) ] &= \frac{1}{2} \langle \lambda + \rho, \check{\alpha_1} \rangle \\ Z(\lambda) [ L(s_1 \cdot 0) ] &= \frac{1}{2} \langle \lambda + \rho, \check{\alpha_2} \rangle \\ Z(\lambda) [ L(s_2 \cdot 0) ] &= \frac{1}{2} \langle \lambda + \rho, \check{\alpha_1} \rangle \qquad \qquad \blacksquare \end{align*} \end{example}

\subsection{Main result} \label{mainres} Now we are ready to state the main result. In the introduction, we briefly described what a ``real variation of stability conditions'' is (as originally defined in Section $1.4$ of \cite{abm}). Now we will elaborate and give a more detailed definition.

\begin{definition} \label{realvar} Let $\mathcal{C}$ be a finite type triangulated category, and let $\Sigma$ be a discrete collection of affine hyperplanes in a finite-dimensional, real vector space $V$. Let $\textbf{Alc}$ (the set of ``alcoves") be the connected components of $V^0 = V \backslash \Sigma$. For each affine hyperplane in $\Sigma$, consider the parallel hyperplane passes through $0$, and let $\Sigma_{lin}$ be the set of those hyperplanes. Fix a component $V^+$ of $V \backslash \Sigma_{lint}$. Given two alcoves $A, A' \in \textbf{Alc}$ which share a co-dimension 1 face, we say that $A'$ is above $A$ if, after we shift the hyperplane so that it passes through $0$, then $A'$ lies on the same side of the hyperplane as $V^+$.

A ``real variation of stability conditions'' on $\mathcal{C}$ consist of a polynomial map $Z: V \rightarrow (K^0(\mathcal{C}) \otimes \mathbb{R})^*$ (known as ``the central charge"), and a map $\tau$ from \textbf{Alc} to the set of bounded $t$-structures on $\mathcal{C}$, such that: \begin{itemize} \item Let $A \in \textbf{Alc}$, and let $M$ be a non-zero object in the heart $\mathcal{A}$ of $\tau(A)$. Then $\langle Z(x), [M] \rangle > 0$ for $x \in A$. \item Let $A' \in \textbf{Alc}$ be another alcove sharing a co-dimension one face $H$ with $A$, and lying above $A$. Let $\mathcal{A}_n$ be the Serre subcategory consisting of objects  $M$ such that the polynomial function $x \rightarrow \langle Z(x), [M] \rangle$ has a zero of order at least $n$ on $H$. Also define $\mathcal{C}_n = \{ C \in \mathcal{C} \; | \; H^i_{\tau(A)}(C) \in \mathcal{A}_n \}$. Then the truncation functors for $\tau(A')$ preserves the filtration by $\mathcal{C}_n$, and the two t-structures on $\mathcal{C}_n / \mathcal{C}_{n+1}$ induced by $\tau(A)$ and $\tau(A')$ differ by a shift of $[n]$. \end{itemize} \end{definition}

\begin{theorem} \label{main} The following datum constitutes an example of ``real variations of stability conditions": \begin{itemize} \item Let $V = \mathfrak{h}^*$, and let $\Sigma$ consist of the co-root hyperplanes $\langle \lambda + \rho, \check{\alpha} \rangle = 0$ (where $\check{\alpha} \in \Delta^+$).  \item The set of alcoves, $\textbf{Alc}$ are naturally identified with the Weyl group $W$; denote by $\underline{w}$ the alcove consisting of $\lambda$, with $\langle w^{-1}(\lambda + \rho), \check{\alpha} \rangle \geq 0$ for all $\alpha \in \Delta^+$.  \item Let $V^+$ be the alcove $\underline{1}$ . \item Let $\mathcal{C} = D^b(\mathcal{O}_0^{d})$. \item Let the central charge $Z: \mathfrak{h}^* \rightarrow (K^0(\mathcal{C}) \otimes \mathbb{R})^*$ be the map constructed in Proposition \ref{Z}. \item Given $w \in W$, let $\tilde{w}$ be its lift to the braid group $\mathbb{B}_W$. Let $\tau(\underline{w})$ be the image of the tautological $t$-structure on $\mathcal{C}$ under the automorphism $\Phi(\tilde{w})$. \end{itemize} \end{theorem} 

\section{Leading coefficients and the central charge}

Given $M \in \mathcal{O}_\lambda$, for each $\mu \in \mathfrak{h}^*$ denote by  $M_{\mu}$ the corresponding weight space. Recall that $$ \text{ch}(M) = \sum_{\mu \in \mathfrak{h}^*} (\text{dim }M_{\mu}) e^{\mu}$$ Using the PBW theorem, $\Delta(w \cdot \lambda)$ $$ \text{ch}( \Delta(w \cdot \lambda)) = \frac{e^{w \cdot \lambda}}{ \prod_{\alpha \in \Delta^+} (1 - e^{-\alpha}) } = \frac{e^{w \rho}}{ \prod_{\alpha \in \Delta^+} (e^{\alpha / 2} - e^{-\alpha / 2}) }$$

Given $M \in \mathcal{O}_{\lambda}$, since the Verma modules $\{ \Delta(w \cdot \lambda) \}_{w \in W}$ form a basis of the Grothendieck group $K^0(\mathcal{O}_{\lambda})$, we have $$ [M] = \sum_{w \in W} a_w [\Delta(w \cdot \lambda)]$$ for some $a_w \in \mathbb{Q}$. Then we can express $$\text{ch}(M) =  \frac{ \sum_{w \in W} a_w e^{w \rho}}{ \prod_{\alpha \in \Delta^+} (e^{\alpha / 2} - e^{-\alpha / 2}) }$$ 

In order to prove the above proposition, we will make use of the following Lemma. Given a module $M \in \mathcal{O}_{\lambda}^d$, the following Lemma tells us how to deduce $\text{GK}(M)$ and $\text{LC}(M)$, knowing the character of $M$.

\begin{proposition} \label{taylor} Suppose $$\text{ch}(M) =  \frac{ \sum_{w \in W} a_w e^{w (\rho+\lambda)}}{ \prod_{\alpha \in \Delta^+} (e^{\alpha / 2} - e^{-\alpha / 2}) }$$ Expand the numerator as a Taylor series, so that we obtain a function, $f_M$ on $\mathfrak{h}^*$. Let $k$ be minimal, such that the degree $k$ component, $f_M^k$, of this polynomial does not vanish. Then, there exists a constant $c$ (depending on $\mathfrak{g}$ and $d$) such that (here $\check{\rho}$ is the half-sum of the positive roots): \begin{align*} \text{GK}(M) &= |\Delta^+| - k \\ \text{LC}(M) &= c f_M^k(\check{\rho}) \end{align*} \end{proposition}

\begin{example} Suppose $M = L(\lambda)$ is the finite-dimensional irreducible module with highest weight $\lambda$, so that by the Weyl character formula $$\text{ch}(M) =  \frac{\sum_{w \in W} \text{sgn}(w) e^{w (\rho+\lambda)}}{\prod_{\alpha \in \Delta^+} (e^{\alpha / 2} - e^{-\alpha / 2}) }$$ It can be shown that for $k < |\Delta^+|$: $$ \sum_{w \in W} \text{sgn}(w) \{ w(\rho + \lambda) \}^k = 0 $$ Thus $k = |\Delta^+|$ is the smallest $k$ for which $f_M^k \neq 0$. Clearly $M$ has Gelfand-Kirillov dimension $0$, so this is consistent with the first claim in Proposition \ref{taylor}.

On the other end of the spectrum, suppose instead that $M$ is the Verma module $\Delta(\lambda)$ for some $\lambda \in \mathfrak{h}^*$. The function $f_M$ is the Taylor expansion of $e^{\lambda+\rho}$, and $k=0$ (since the degree $0$ component of $f_M$ is non-zero). By Lemma \ref{gkdim}, $M$ has Gelfand-Kirillov dimension $|\Delta^+|$; again, this is consistent with the first claim in Proposition \ref{taylor}. $\blacksquare$ \end{example}

\begin{example} Now let us verify Proposition \ref{taylor} for the $4$ simple modules lying in $\mathcal{O}_0^1$ for $\mathfrak{g} = \mathfrak{sl}_3$. From the discussion in Example \ref{ex1}, it follows that: \vspace{10mm} \begin{align*} \text{ch } L((s_1 s_2) \cdot \lambda) &= \text{ch } \Delta((s_1 s_2) \cdot \lambda) - \text{ch } \Delta((s_1 s_2 s_1) \cdot 0) \\ \text{ch } L((s_2 s_1) \cdot \lambda) &= \text{ch } \Delta((s_2 s_1) \cdot \lambda) - \text{ch } \Delta( (s_1 s_2 s_1) \cdot \lambda) \\ \text{ch } L(s_1 \cdot \lambda) &= \text{ch } \Delta(s_1 \cdot \lambda) - \text{ch } \Delta((s_1 s_2) \cdot \lambda) - \text{ch } \Delta((s_2 s_1) \cdot \lambda) + \text{ch } \Delta((s_1 s_2 s_1) \cdot \lambda) \\ \text{ch } L(s_2 \cdot \lambda) &= \text{ch } \Delta(s_2 \cdot \lambda) - \text{ch } \Delta((s_1 s_2) \cdot \lambda) - \text{ch } \Delta((s_2 s_1) \cdot \lambda) + \text{ch } \Delta((s_1 s_2 s_1) \cdot \lambda) \end{align*} \vspace{5mm}
For each of these modules $M$, clearly $f_M^0 = 0$, and $f_M^1$ is as follows: \begin{align*} M &= L((s_2 s_1) \cdot 0), f_M^1 = \langle \lambda + \rho, \check{\alpha_2} \rangle \alpha_1 \\ M &= L((s_1 s_2) \cdot 0), f_M^1 = \langle \lambda + \rho, \check{\alpha_1} \rangle \alpha_2 \\ M &= L(s_1 \cdot 0), f_M^1 = \langle \lambda + \rho, \check{\alpha_2} \rangle \alpha_2 \\ M &= L(s_2 \cdot 0), f_M^1 = \langle \lambda + \rho, \check{\alpha_1} \rangle \alpha_1 \end{align*} Using the calculations in Example \ref{ex3}, it is clear these that these computations are consistent with Proposition \ref{taylor}.  $\blacksquare$ \end{example}

\begin{proof}[Proof of Proposition \ref{taylor}] We may make the assumption that $M$ is a highest weight module, i.e. it is a quotient of $\Delta(\lambda)$ for some $\lambda$. To see this, note that the functions $M \rightarrow \text{LC}(M)$ and $M \rightarrow f_M^k(\check{\rho})$ are additive on exact triangles $0 \rightarrow M_1 \rightarrow M \rightarrow M_2 \rightarrow 0$ such that $\text{GK}(M_1)=\text{GK}(M_2) = \text{GK}(M)$. Thus we can choose a Jordan-Holder filtration of $M$, where each simple sub-quotient has the same Gelfand-Kirillov dimension; the conclusion would then follow if we knew it to be true for all highest weight modules, since all simples fall into this category.

So let $M_0 = \mathbb{C} v_{\lambda}$, for some highest weight vector $v_{\lambda} \in M$. Given $\mu \in \mathfrak{h}^*$ such that $M_{\mu} \neq 0$, define $d(\mu) = \langle \check{\rho}, \lambda - \mu \rangle$; alternatively $d(\mu) = \sum_{i \in I} d_i$ if $\lambda - \mu = \sum_{i \in I} d_i \alpha_i$ (note that $d_i \geq 0$, since $M$ is a quotient of $\Delta(\lambda)$). Then: \begin{align*} M_i &= \bigoplus_{d(\mu) \leq i} M_{\mu} \\ \text{dim } M_i &= \sum_{d(\mu) \leq i} \text{dim } M_{\mu} \end{align*} To prove this proposition, we will explicitly compute both sides of the below equality. Let us start with the LHS. \begin{align*} (\sum_{\mu \in \Lambda^+} (\text{dim }M_{\mu}) e^{\mu}) (e^{t \check{\rho}}) \prod_{\alpha \in \Delta^+} (e^{\frac{\alpha}{2}} - e^{\frac{-\alpha}{2}}) (e^{t \check{\rho}}) &= \sum_{w \in W} d_w e^{w(\lambda + \rho)} (e^{t \check{\rho}}) \\ (\sum_{\mu \in \Lambda^+} (\text{dim }M_{\mu}) e^{\mu}) e^{t \check{\rho}} &= e^{t \langle \lambda, \check{\rho} \rangle} \sum_{n \geq 0} \text{dim} (M_n / M_{n-1}) e^{-nt} \end{align*} For $n$ sufficiently large, we have $\text{dim } M_n = p_M(n)$ for some polynomial $p_M$; thus $\text{dim}(M_n / M_{n-1}) = q_M(n)$, where $q_M(x) = p_M(x) - p_M(x-1)$. By differentiating the identity $1 + s + s^2 + \cdots = (1-s)^{-1}$ repeatedly, we obtain that: \begin{align*} \sum_n n(n-1) \cdots (n-k+1) s^{n-k} &= k! (1-s)^{-k-1} \end{align*} By taking linear combinations of the above identity, we deduce that there exists a polynomial $\tilde{q}_M$ with degree $d$, leading coefficient $d! \text{LC}(M)$ and no constant term, satisfying the following. Then we will evaluate at $s=e^{-t}$, and continue with the computation of the left hand side (here the polynomial $r_M$ accounts for the fact that $\text{dim }M_n \neq p(n)$ at finitely many values, and $C$ is some constant): \begin{align*} \sum_{n} q_M(n) s^n &= \tilde{q}_M ((1-s)^{-1}) \\ \sum_{n \geq 0} \text{dim} (M_n / M_{n-1}) e^{-nt} &= \tilde{q}_M((1-e^{-t})^{-1}) + r_M(e^{-t}) \end{align*} \begin{align*} \prod_{\alpha \in \Delta^+} (e^{\frac{\alpha}{2}} - e^{\frac{-\alpha}{2}})(e^{t \check{\rho}}) &= \prod_{\alpha \in \Delta^+}  (e^{t \langle \check{\rho}, \frac{\alpha}{2} \rangle} - e^{- t \langle \check{\rho}, \frac{\alpha}{2} \rangle})  \\  &= \prod_{\alpha \in \Delta^+} (t \langle \check{\rho}, \alpha \rangle + \frac{t^3 {\langle \check{\rho}, \alpha \rangle}^3}{24} + \cdots) \\ &= \{ \prod_{\alpha \in \Delta^+} \langle \check{\rho}, \alpha \rangle \}  t^{|\Delta^+|} (1 + t^2 C + \cdots) \end{align*} \begin{align*} \text{LHS} &= e^{t \langle \lambda, \check{\rho} \rangle} \{ \tilde{q}_M((1-e^{-t})^{-1}) + r_M(e^{-t}) \} \{ \prod_{\alpha \in \Delta^+} \langle \check{\rho}, \alpha \rangle \}  t^{|\Delta^+|} (1 + t^2 C + \cdots)  \\ &= \{ d! \text{ LC(M)} t^{-d} (1 + \cdots) \} \{ \prod_{\alpha \in \Delta^+} \langle \check{\rho}, \alpha \rangle \}  t^{|\Delta^+|} (1 + t \langle \lambda, \check{\rho} \rangle + \cdots) \\ &= d! \text{ LC(M)} \prod_{\alpha \in \Delta^+} \langle \check{\rho}, \alpha \rangle \; t^{|\Delta^+| - d} (1 + \cdots) \end{align*} Above, we have used the Taylor expansion $(1-e^{-t})^{-1} = t^{-1}(1 + \frac{t}{2} + \cdots)$ to obtain the leading coefficient of the Taylor expansion of $\tilde{q}_M((1-e^{-t})^{-1})$. \begin{align*} \sum_{w \in W} d_w e^{w(\lambda + \rho)} (e^{t \check{\rho}}) &= \sum_{w \in W} d_w e^{\langle w(\lambda + \rho), t \check{\rho} \rangle} \\ &= \sum_{w \in W, n \geq 0} d_w \frac{t^n}{n!} {\langle w(\lambda + \rho), \check{\rho} \rangle}^n \end{align*} Comparing the Taylor expansions of the LHS and the RHS, it follows that $\text{LC}(M) = c f_M^k(\check{\rho})$ for some constant $c$ and $k = |\Delta^+| - d$; and further $f_M^{i}(\check{\rho}) = 0$ for $i < k$. Now pick $\check{\rho}'$ arbitrary satisfying $\langle \check{\rho}', \alpha \rangle \in \mathbb{Z}_{>0}$ for simple roots $\alpha$. Repeating the whole argument, we deduce that $f_M^{i}(\check{\rho}') = 0$ for any such $\check{\rho}'$. It then follows that $f_M^i = 0$ for $i<k$, completing the proof. \end{proof}

\begin{proof}[Proof of Proposition \ref{Z}] Suppose that $$ \text{ch}(M) =  \frac{ \sum_{w \in W} a_w e^{w \rho}}{ \prod_{\alpha \in \Delta^+} (e^{\alpha / 2} - e^{-\alpha / 2})}$$ Then, since $T_{0 \rightarrow \lambda} \Delta(w \cdot 0) = \Delta(w \cdot \lambda)$: $$\text{ch}(T_{0 \rightarrow \lambda} M) =  \frac{ \sum_{w \in W} a_w e^{w (\rho+\lambda)}}{ \prod_{\alpha \in \Delta^+} (e^{\alpha / 2} - e^{-\alpha / 2}) }$$ It is clear that $\text{GK}(T_{0 \rightarrow \lambda} M) = \text{GK}(M) = d$. Let us define the polynomial $$Z(\lambda)([M]) = c \sum_{w \in W} a_w \{(w (\rho + \lambda)(\check{\rho})\}^{|\Delta^+| - d}$$ Using Proposition \ref{taylor}, it follows that $\text{LC}(T_{0 \rightarrow \lambda} M) = Z(\lambda)([M])$ for $\lambda \in \Lambda^+$, proving the first claim. The uniqueness of the polynomial is clear once its values on the lattice $\Lambda^+$ are specified.

Next, using Lemma \ref{WGroth}, we have: \begin{align*} \text{ch}(\Phi(\widetilde{y}) M) &= \frac{ \sum_{w \in W} a_w e^{(wy) \rho}}{ \prod_{\alpha \in \Delta^+} (e^{\alpha / 2} - e^{-\alpha / 2})} \\ &= \frac{ \sum_{w \in W} a_{wy^{-1}} e^{w \rho}}{ \prod_{\alpha \in \Delta^+} (e^{\alpha / 2} - e^{-\alpha / 2})} \\ Z(y^{-1} \cdot \lambda)([\Phi(\widetilde{y}) M]) &= c \sum_{w \in W} a_{wy^{-1}} \{w (\rho + y^{-1} \cdot \lambda)(\check{\rho})\}^{|\Delta^+| - d} \\ &= c \sum_{w \in W} a_{wy^{-1}} \{wy^{-1} (\lambda + \rho)(\check{\rho})\}^{|\Delta^+| - d} \\ &= c \sum_{w \in W} a_{w} \{w (\lambda + \rho)(\check{\rho})\}^{|\Delta^+| - d} = Z(\lambda)([M]) \end{align*} This proves the statement of the second claim. \end{proof}

\begin{remark} \label{firstpart} At this point, we are almost ready to check the first condition involving real variations of stability conditions. Suppose $M$ lies in the heart of the t-structure $\tau(w)$ (i.e. $M = \Phi(\tilde{w})(M')$ for some $M' \in \mathcal{O}^0_d$), and that $\lambda$ lies inside the alcove $\underline{w}$ (i.e. $\lambda = w \cdot \lambda'$ for some $\lambda'$ lying inside the alcove $\underline{1}$). Then, using Proposition \ref{Z}: \begin{align*} \langle Z(\lambda)[M] \rangle &= Z(w \cdot \lambda') [\Phi(\tilde{w})(M')] \\ &= Z(\lambda')[M'] \end{align*} Now if $\lambda' \in \Lambda^+$, then $Z(\lambda')[M'] = \text{LC}(T_{0 \rightarrow \lambda'} M') \geq 0$ (by definition of leading coefficient). However, to rigorously show that this statement is true for arbitrary $\lambda'$ in the fundamental Weyl chamber, we will need more machinery. \end{remark} 

\section{Proof of Main Results} \label{3}

\subsection{Two-step filtrations.} In this section we will check that the second condition from the above Definition is satisfied.

Suppose $\underline{w}$ and $\underline{w'}$ are two adjacent alcoves separated by a hyperplane $H$ (i.e. $w' = s_{\alpha} w$, where $\alpha$ is a simple root); and suppose that $\underline{w'}$ lies above $\underline{w}$. Recall that $\mathcal{A}_{\underline{w}}$ denotes the heart of the t-structure $\tau(\underline{w})$. Denote by $\mathcal{A}_{\underline{w}, \underline{w'}}^n$ (resp. $\mathcal{A}_{\underline{w'}, \underline{w}}^n$) the sub-category consisting of objects $M \in \mathcal{A}_{\underline{w}}$ (resp. $M \in \mathcal{A}_{\underline{w'}}$) such that the function $f_M: \mathfrak{h}^* \rightarrow \mathbb{C}$ defined by $f_M(x) = \langle Z(x), [M] \rangle$ has a zero of order at least $n$ on $H$. The following Lemma allows us to reduce to the case where $w=1$, and  the two propositions give a very concrete descriptions of these sub-categories in that case.  

\begin{lemma} \label{easy} \begin{enumerate} \item $\mathcal{A}_{\underline{w'}, \underline{w}}^k = \Phi(\widetilde{w}) \mathcal{A}_{\underline{s_{\alpha}}, \underline{1}}^k$, and $\mathcal{A}_{\underline{w}, \underline{w'}}^k = \Phi(\widetilde{w}) \mathcal{A}_{\underline{1}, \underline{s_{\alpha}}}^k$ \item $\mathcal{A}_{\underline{w}, \underline{w'}}^k = \Phi(\widetilde{s_{\alpha}}) \mathcal{A}_{\underline{w'}, \underline{w}}^k$ (in particular, $\mathcal{A}_{\underline{s_{\alpha}}, \underline{1}}^k = \Phi(\widetilde{s_{\alpha}}) \mathcal{A}_{\underline{k}, \underline{s_{\alpha}}}^1$) \end{enumerate} \end{lemma}

 \begin{proposition} \label{3.1}   The category $\mathcal{A}_{\underline{1}, \underline{s_{\alpha}}}^1 \subset \mathcal{A}_{\underline{1}} = \mathcal{O}_0^{d}$ consists of those objects $M \in \mathcal{O}_0^{d}$ which possess a filtration, with each quotient being a simple modules $L(w \cdot 0)$ with $l(w s_{\alpha}) = l(w) + 1$. \end{proposition}

\begin{proposition} \label{3.2} For $n \geq 2$, the categories $\mathcal{A}_{\underline{w}, \underline{w'}}^n = \{ 0 \}$. \end{proposition}

\begin{example} Before proving the above two propositions, let us re-visit our running example (with $\mathfrak{g} = \mathfrak{sl}_3$ and $d=1$) and verify them by hand in that case. From the calculations in Example \ref{ex3}, we deduce that the category $\mathcal{A}^1_{\underline{1}, \underline{s_1}}$ consists of modules in $\mathcal{O}_0^1$ which have a filtration whose sub-quotients are $L((s_1 s_2) \cdot 0)$ or $L(s_2 \cdot 0)$; and that the category $\mathcal{A}^1_{\underline{1}, \underline{s_2}}$ consists of modules in $\mathcal{O}_0^1$ which have a filtration whose sub-quotients are $L((s_2 s_1) \cdot 0)$ or $L(s_1 \cdot 0)$. This is consistent with Proposition \ref{3.1}. For each $M \in \mathcal{O}_0^1$, $f_M: \mathfrak{h}^* \rightarrow \mathbb{C}$ is a linear function, and hence cannot have a double zero on any hyperplane (unless $M=0$); this is consistent with Proposition \ref{3.2}. $\blacksquare$ \end{example} \begin{proof}[Proof of Lemma \ref{easy}] First let us prove $(1)$. Given an object $M \in \mathcal{C}$, recall that $f_M$ denotes the function on $\mathfrak{h}^*$ defined by $f_M(x) = \langle Z(x), [M] \rangle$. Denote by $\text{deg}(f_M|_H)$ the order of vanishing of the polynomial $f_M$ on the hyperplane $H$. \begin{align*} \mathcal{A}_{\underline{w}, \underline{w'}}^k &= \{ M \in \mathcal{A}_w \; | \; \text{deg}(f_M|_{H}) \geq k \} \\ &= \{ \Phi(\widetilde{w})N , N \in \mathcal{O}_0^d  \; | \; \text{deg}(f_{\Phi(\widetilde{w})N}|_{H}) \geq k \} \\ &= \{ \Phi(\widetilde{w})N , N \in \mathcal{O}_0^d  \; | \; \text{deg}(f_{N}|_{H^{\alpha}}) \geq k \} \\ &= \Phi(\widetilde{w}) \mathcal{A}_{\underline{1}, \underline{s_{\alpha}}}^k \end{align*} Here $H^{\alpha}$ denotes the wall $\langle \lambda + \rho, \check{\alpha} \rangle = 0$ separating $\Lambda^+$ and $s_{\alpha} \cdot \Lambda^+$. Above we have used the fact that $\text{deg}(f_{\Phi(\widetilde{w})N}|_{H}) = \text{deg}(f_{N}|_{H^{\alpha}})$, which follows from $W$-equivariance. Similarly, it follows that $\mathcal{A}_{\underline{w'}, \underline{w}}^k = \Phi(\widetilde{w}) \mathcal{A}_{\underline{s_{\alpha}}, \underline{1}}^k$

Using $(1)$, in order to show $(2)$ it suffices to prove that $\mathcal{A}_{\underline{s_{\alpha}}, \underline{1}}^k = \Phi(\widetilde{s_{\alpha}}) \mathcal{A}_{\underline{1}, \underline{s_{\alpha}}}^1$. This is clear, since: \begin{align*} \mathcal{A}_{\underline{s_{\alpha}}, \underline{1}}^k &= \{ \Phi(\widetilde{s_{\alpha}}) C, C \in \mathcal{O}_0^d \; | \; \text{deg}(f_{\Phi(\widetilde{s_{\alpha}}) C} |_{H^{\alpha}}) \geq k \} \\ &= \{ \Phi(\widetilde{s_{\alpha}}) C, C \in \mathcal{O}_0^d \; | \; \text{deg}(f_{C}|_{H^{\alpha}}) \geq k \} \\ &= \Phi(\widetilde{s_{\alpha}}) \mathcal{A}_{\underline{1}, \underline{s_{\alpha}}}^1 \\ \end{align*} Above, $\text{deg}(f_{\Phi(\widetilde{s_{\alpha}}) C} |_{H^{\alpha}}) =\text{deg}(f_{C} |_{H^{\alpha}})$ using $W$-equivariance (since $s_{\alpha}$ acts via reflecting about the hyperplane $H^{\alpha}$). \end{proof}

The following Lemma is well-known; see for instance Theorem 5.1 in Andersen-Stroppel's article \cite{as}. The notation used there is slightly different, so we include a proof for the reader's convenience.

\begin{lemma} \label{weak} Suppose $M$ has a filtration, with each successive quotient being a simple module $L(w \cdot 0)$ with $l(w s_{\alpha}) = l(w) + 1$. Then $\Phi(\widetilde{s_{\alpha}}) M = M[1]$ \end{lemma} \begin{proof} Fix $\mu$, such that $\langle \mu + \rho, \check{\alpha} \rangle = 0, \langle \mu + \rho, \check{\beta} \rangle = 0$. By definition, we need to show that $\text{Cone}(M \rightarrow R_{\alpha} M) = M[1]$ (where $R_{\alpha} = T_{\mu \rightarrow 0} T_{0 \rightarrow \mu}$). So it suffices to show that $R_{\alpha} M = 0$. We will show the stronger statement that $T_{0 \rightarrow \mu} M = 0$. 

Suppose that $l(w s_{\alpha}) = l(w) + 1$. Then we have a sequence of maps: $$ \Delta(w s_{\alpha} \cdot 0) \overset{i}{\rightarrow} \Delta(w \cdot 0)  \overset{p}{\rightarrow} L(w \cdot 0) $$ The map $i$ is injective, and its existence follows using Proposition $1.4$ of \cite{h};  the map $p$ is clearly surjective. Since the image of $i$ lands inside the maximal submodule of $\Delta(w \cdot 0)$, the composition of these two maps is $0$. Applying the translation functor $T_{0 \rightarrow \mu}$ to this triangle, we get: $$ \Delta(w s_{\alpha} \cdot \mu) \overset{i_T}{\rightarrow} \Delta(w \cdot \mu)  \overset{p_T}{\rightarrow} T_{0 \rightarrow \mu} L(w \cdot 0) $$ However, since $\langle \mu + \rho, \check{\alpha} \rangle = 0$, it follows that $s_{\alpha} \cdot \mu = \mu$; and hence the map $i_T$ is an isomorphism. Since the composition of the two maps is $0$, and $p_T$ is surjective, it follows that $T_{0 \rightarrow \lambda} L(w \cdot 0) = 0$.

Now suppose $M$ has a filtration by such modules $L(w \cdot 0)$. By using the exact-ness of the functor $T_{0 \rightarrow \lambda}$, and inducting on the length $l(M)$ of this filtration, it follows that $T_{0 \rightarrow \lambda} M = 0$. \end{proof}

\begin{proof}[Proof of Proposition \ref{3.1}] First we prove that any module $M$, with such a filtration lies inside $\mathcal{A}_{\underline{1}, \underline{s_{\alpha}}}^1$.  Suppose $\langle x + \rho, \alpha \rangle=0$; then $s_{\alpha} \cdot x = x$, so: \begin{align*} \langle Z(x), [M] \rangle &= \langle Z( s_{\alpha} \cdot x), [\Phi(s_{\alpha}) M] \rangle \\  &= \langle Z(x), [M[1]] \rangle \\  &= - \langle Z(x), [M] \rangle \\ \Rightarrow \langle Z(x), [M] \rangle &= 0 \end{align*} Then we have that, $\langle Z(x), [M] \rangle = 0$ if $\langle x + \rho, \check{\alpha} \rangle = 0$, and so $M \in \mathcal{A}_{\underline{1}, \underline{s_{\alpha}}}^1$ by definition.

Next let us prove that any module $M \in \mathcal{A}_{\underline{1}, \underline{s_{\alpha}}}^1$ possesses such a filtration. Suppose the Jordan-Holder filtration of $M$ contains the simple module $L(w \cdot 0)$ with multiplicity $m_w$, so that \begin{align*} \langle Z(x), [M] \rangle &= \langle Z(x), \sum_{w \in W} m_w [L (w \cdot 0)] \rangle \\ &= \sum_{w \in W} m_w \langle Z(x), [L(w \cdot 0)] \rangle \end{align*} Now pick $\lambda$ to be integral such that $\langle \lambda + \rho, \check{\alpha} \rangle = 0$, but $\langle \lambda + \rho, \check{\beta} \rangle > 0$ for all simple roots $\beta \neq \alpha$ (so $\lambda$ lies in the closure of the dominant alcove). Then we claim that $\langle Z(\lambda), L(w \cdot 0) \rangle > 0$ if $l(w s_{\alpha}) = l(w) - 1$.  The desired result would follow. 

 It is known (for instance, see Section $7.7$ of \cite{h}) that $T_{0 \rightarrow \lambda}$ sends an irreducible module either to $0$, or to another irreducible; by counting the number of irreducibles it follows that $T_{0 \rightarrow \lambda} L(w \cdot 0) = L(w \cdot \lambda)$. By using the techniques employed in the proof of Proposition \ref{taylor}, it follows that $L(w \cdot \lambda)$ and $L(w \cdot 0)$ have the same Gelfand-Kirillov dimension. Further, we deduce that $\langle Z(\lambda), L(w \cdot 0) \rangle = \text{LC}(L(w \cdot \lambda)) >0$, as required. \end{proof}  %It is clear that $\text{GK}(T_{0 \rightarrow \lambda} L(w \cdot 0)) \leq \text{GK}(L(w \cdot 0))$, since tensoring by a finite dimensional and projecting to a block cannot increase the Gelfand-Kirillov dimension. 
% \end{proof}To prove this, we claim that $T_{0 \rightarrow \lambda} L(w \cdot 0)$ is a non-zero module with Gelfand-Kirillov dimension $d$, and that $\langle Z(\lambda), L(w \cdot 0) \rangle = \text{LC}(T_{0 \rightarrow \lambda} L(w \cdot 0))$; the desired result would follow.

In order to prove Proposition \ref{3.2}, we appeal to the theory of harmonic polynomials. For a detailed exposition, we refer the reader to Section $6.3$ and $6.4$ of \cite{cg}. 

\begin{definition} A polynomial function $f: \mathfrak{h}^* \rightarrow \mathbb{C}$ is a ``harmonic polynomial"' if for every $\partial \in \mathcal{D}(\mathfrak{h}^*)_{+}^{W}$, $\partial f = 0$. \end{definition}

\begin{proposition} \label{taylorharmonic} Fix an object $M \in \mathcal{C}$. The function $f_M: \mathfrak{h}^* \rightarrow \mathbb{C}$, given by $f_M(x) = \langle Z(x), [M] \rangle$ is a harmonic polynomial. \end{proposition} 

 \begin{proposition} \label{nodoublezero} Any non-zero harmonic polynomial cannot have a double zero on a co-root hyperplane. \end{proposition} 

\begin{example} Before proving these two propositions, let us revisit the example $\mathfrak{g} = \mathfrak{sl}_3$. First let us calculate the invariant differential operators, and the harmonic polynomials in this case. Define $X_1, X_2: \mathfrak{h}^* \rightarrow \mathbb{C}$ by setting $X_1(\lambda) = \langle \lambda, \check{\alpha_1} \rangle, X_2(\lambda) = \langle \lambda, \check{\alpha_2} \rangle$; then the set of polynomial functions from $\mathfrak{h}^*$ to $\mathbb{C}$ can be naturally identified by $\mathbb{C}[X_1, X_2]$. We compute that: \begin{align*} s_1 \frac{\partial}{\partial X_1} &= - \frac{\partial}{\partial X_1} + \frac{\partial}{\partial X_2}, \qquad s_2\frac{\partial}{\partial X_1} = \frac{\partial}{\partial X_1} \\ s_1\frac{\partial}{\partial X_2} &= \frac{\partial}{\partial X_2}, \qquad \qquad s_2 \frac{\partial}{\partial X_2} = \frac{\partial}{\partial X_1} - \frac{\partial}{\partial X_2} \\ \mathcal{D}(\mathfrak{h}^*)^W_{+} &= \langle (\frac{\partial}{\partial X_1})^2 + (\frac{\partial}{\partial X_2})^2 - \frac{\partial}{\partial X_1} \frac{\partial}{\partial X_2}, (\frac{\partial}{\partial X_1})^2 \frac{\partial}{\partial X_2} - \frac{\partial}{\partial X_1} (\frac{\partial}{\partial X_2})^2 \rangle \end{align*} So the space $\mathcal{H}$ of harmonic polynomials are those annihilated by those two polynomials: \begin{align*} \mathcal{H} = \mathbb{C} \{ X_1^2 X_2 + X_1 X_2^2, X_1^2 + 2 X_1 X_2, X_2^2 + 2X_1 X_2, X_1, X_2, 1 \} \end{align*} This is consistent with Proposition \ref{taylorharmonic}, since when $M$ is a simple module in $\mathcal{O}_0^1$, $f_M(x) = X_1 + 1$ or $f_M(x) = X_2 +1$; in either case $f_M \in \mathcal{H}$. It is easy to verify that no element of $\mathcal{H}$ is divisible by $(X_1 + 1)^2$ or $(X_2+1)^2$; this is consistent with Proposition \ref{nodoublezero}. $\blacksquare$ \end{example}

\begin{proof}[Proof of Proposition \ref{taylorharmonic}] This follows from Proposition \ref{taylor}, combined with the following Lemma. \end{proof}  

\begin{lemma} Given a collection $\{ a_w \}_{w \in W}$ of complex numbers, let $d$ be minimal such that $$ R_d := \sum_{w \in W} a_w  \frac{ {w(\rho + \lambda)}^d}{d!}$$ is a non-zero function. Then $R_d$ is a harmonic polynomial. \end{lemma} \begin{proof} From Proposition $6.4.4$ in \cite{cg}, it follows that $R_d(\lambda - \rho)$ is a harmonic polynomial. The result now follows using the well-known fact that any harmonic polynomial is stable under shifts. \end{proof}

\begin{proof}[Proof of Proposition \ref{nodoublezero}] See the last paragraph of Proposition $1$ in \cite{abm} (on page $9$). \end{proof} 

\begin{proof}[Proof of Proposition \ref{3.2}] This follows using Propositions \ref{taylorharmonic} and \ref{nodoublezero}. \end{proof} 

Before returning to the proof of the Main Theorem, we will need the following three Lemmas (the first of which is a strengthening of Lemma \ref{weak}).

\begin{lemma} \label{strong} If $A \in \mathcal{C}$ satisfies $H^n(A)  \in \mathcal{A}^1_{\underline{1}, \underline{s}_{\alpha}} \; \forall \; n \in \mathbb{Z}$, then $\Phi(\tilde{s}_{\alpha})(A) \simeq A[1]$. \end{lemma} \begin{proof} Given an object $A \in \mathcal{C}$, define the length $l(A) = | \{ i \in \mathbb{Z}, H^i(A) \neq 0 \} |$. Let us proceed by induction on $l(A)$. If $l(A) = 1$, the statement follows from Proposition $3.1$ and Lemma $3.4$. 

Now suppose that $l(A) = i$. We can pick $j$ so that $l(\tau_{\leq j} A), l(\tau_{\geq j+1} A) < i$ (here $\tau$ denotes truncation with respect to the standard $t$-structure on $\mathcal{C}$). Applying the automorphism $\Phi(\tilde{s}_{\alpha})(A)$, and using the induction hypothesis, we have distinguished triangles: \begin{align*} \tau_{\leq j} A &\rightarrow A \rightarrow \tau_{\geq j+1} A \\ \Phi(\tilde{s}_{\alpha}) \tau_{\leq j} A &\rightarrow \Phi(\tilde{s}_{\alpha}) A \rightarrow \Phi(\tilde{s}_{\alpha}) \tau_{\geq j+1} A \\ \tau_{\leq j} A[1] &\rightarrow \Phi(\tilde{s}_{\alpha}) A \rightarrow \tau_{\geq j+1} A[1] \end{align*} Now using the axioms of a triangulated category, it follows that $\Phi(\tilde{s}_{\alpha}) A \simeq A[1]$, as required. \end{proof}

We expect that the following result, in some form, is known to experts; however, we were unable to find it in the literature. 

\begin{lemma} \label{szero} \begin{itemize} \item Any $A \in \mathcal{C}$ satisfies $\Phi(\tilde{s}_{\alpha}) A \simeq A \text{ (mod } \mathcal{C}^1_{\underline{1}, \underline{s_{\alpha}}})$. \item Suppose $A$ is a simple, and $A \notin \mathcal{C}^1_{\underline{1}, \underline{s_{\alpha}}}$. Then we have a map $A \rightarrow \Phi(\tilde{s}_{\alpha}^{-1}) A$, with $A$ being the only simple sub-module of $\Phi(\tilde{s}_{\alpha}^{-1}) A$. \end{itemize} \end{lemma} \begin{proof} It is sufficient to prove the first statement when $A$ is a simple module, since the general case would then follows using the argument used in Lemma \ref{strong}; so let us make that assumption. 

Recall that $\Phi(\tilde{s}_{\alpha}) A = \text{Cone}(A \rightarrow  T_{\mu \rightarrow 0} T_{0 \rightarrow \mu} A)$, where $\mu$ satisfies $$\langle \mu + \rho, \check{\alpha} \rangle = 0, \langle \mu + \rho, \check{\beta} \rangle > 0 \text{ if } \beta \in \Delta^+, \beta \neq \alpha$$ Since the functors $T_{\mu \rightarrow 0}$ and $T_{0 \rightarrow \mu}$ are bi-adjoint, we have natural maps $T_{\mu \rightarrow 0} T_{0 \rightarrow \mu} A \rightarrow A$ and $A \rightarrow T_{\mu \rightarrow 0} T_{0 \rightarrow \mu} A$. We claim that the composition of the two maps $A \rightarrow T_{0 \rightarrow \mu} T_{\mu \rightarrow 0} A \rightarrow A$ is zero. Suppose that it isn't; then the composite map is an isomorphism. Further, neither of the maps can be an isomorphism, since the action of the braid group element $\Phi(\tilde{s}_{\alpha})$ is invertible. Thus $M = T_{\mu \rightarrow 0} T_{0 \rightarrow \mu} A$ contains $A$ as a direct summand. But using adjointness, $\text{Hom}(M, L)$ is $1$-dimensional if $L \simeq A$, and is $0$ otherwise; this is a contradiction.

Since the composition of the two maps is $0$, we have a map $\Phi(\tilde{s}_{\alpha}) A \rightarrow A$, and a map $A \rightarrow \Phi(\tilde{s}_{\alpha}^{-1}) A$. From Corollary $7.12$ in \cite{h}, we have: $$ T_{0 \rightarrow \mu} T_{\mu \rightarrow 0} (T_{0 \rightarrow \mu} A) \simeq T_{0 \rightarrow \mu} A \oplus T_{0 \rightarrow \mu} A$$ It follows that $T_{0 \rightarrow \mu}[\Phi(\tilde{s}_{\alpha}) A]  = T_{0 \rightarrow \mu} A$, and hence $T_{0 \rightarrow \mu}[\Phi(\tilde{s}_{\alpha}) A \rightarrow A] = 0$. From Lemma \ref{strong}, we deduce that if $T_{0 \rightarrow \mu} C = 0$ then $C \in \mathcal{C}^1_{\underline{1}, \underline{s_{\alpha}}}$; this concludes the proof of the first statement.

For the second statement, let $B \neq A$ be another simple module. Then the conclusion follows from the following calculation, since in this case, $\Phi(\tilde{s}_{\alpha}^{-1}) A \neq 0$. $$\text{Hom}(B, \Phi(\tilde{s}_{\alpha}^{-1}) A) \subseteq \text{Hom}(B, T_{\mu \rightarrow 0} T_{0 \rightarrow \mu} A) = \text{Hom}(T_{0 \rightarrow \mu} B, T_{\mu \rightarrow 0}A) = 0$$ \end{proof}

\subsection{Proof of Theorem 1}

\begin{lemma} \label{positive} Suppose $p \in \mathbb{R}[x_1, \cdots, x_n]$ is a homogeneous polynomial, such that $p(x_1, x_2, \cdots, x_n) \geq 0$ if $(x_1, \cdots, x_n) \in \mathbb{Z}_{\geq 0}^n$. Then $p(x_1, x_2, \cdots, x_n) \geq 0$ if $(x_1, \cdots, x_n) \in \mathbb{R}_{\geq 0}^n$. \end{lemma} \begin{proof} Assume to the contrary that $p(x_1, x_2, \cdots, x_n) < 0$ for some $(x_1, \cdots, x_n) \in \mathbb{R}_{\geq 0}^n$. By continuity, we can pick a small open ball $B$ containing $(x_1, \cdots, x_n)$, such that $p(x_1, x_2, \cdots, x_n) < 0$ for all $(x_1, \cdots, x_n) \in B$. Let $\widetilde{B} = \{ t \cdot x \; | \; x \in B, t \in \mathbb{R}_{\geq 0} \}$; since $p$ is a homogeneous polynomial, it follows that $p(x_1, x_2, \cdots, x_n) < 0$ for all $(x_1, \cdots, x_n) \in \widetilde{B}$. However, it is clear that $\widetilde{B}$ contains points in $\mathbb{Z}_{\geq 0}^n$; this contradicts our initial assumption. \end{proof}

\begin{proof}[Proof of Main Theorem] First let us complete the proof of the first condition in Definition \ref{realvar} (which we started in Remark \ref{firstpart}). From the discussion in that remark, it suffices to prove that $Z(\lambda)[M] \geq 0$, where $\lambda \in \underline{1}$, and $M \in \mathcal{O}^0_d$.  We know that $Z(\lambda)[M] = \text{LC}(T_{0 \rightarrow \lambda} M) \geq 0$ if $\lambda \in \Lambda^+$; from the discussion in the proof of Proposition \ref{3.1}, we know that that this statement also holds if $\lambda$ is an integral weight lying inside the closure of the alcove $\underline{1}$. 

Let $\{ \check{\alpha_1}, \cdots, \check{\alpha_r} \}$ are the simple roots, $\{ \Lambda_1, \cdots, \Lambda_r \}$ the corresponding fundamental weights, and suppose that $w^{-1} \check{\alpha}_i = \sum_{1 \leq j \leq r} c^w_{i, j} \check{\alpha_j}$. By Proposition \ref{taylor}, if $[M] = \sum_{w \in W} a_w [\Delta(w \cdot 0)]$ and $k = |\Delta^+| - d$, then: \begin{align*} Z(\lambda)[M] &= \frac{c}{k!} \sum_{w \in W} a_w \{ w(\rho + \lambda) (\check{\rho}) \}^k \\ &= \frac{c}{k!} \sum_{w \in W} a_w \{ \sum_{1 \leq j \leq r} \langle w(\lambda + \rho), \check{\alpha_j} \rangle \Lambda_j (\check{\rho}) \}^k \\ &= \frac{c}{k!} \sum_{w \in W} a_w \{ \sum_{1 \leq j \leq r} \langle \lambda + \rho, w^{-1} \check{\alpha_j} \rangle \Lambda_j (\check{\rho}) \}^k \\ &= \frac{c}{k!} \sum_{w \in W} a_w \{ \sum_{1 \leq l \leq r} \langle \lambda + \rho, \check{\alpha_l} \rangle (\sum_{1 \leq j \leq r} c^w_{j, l} \Lambda_j(\check{\rho}) ) \}^k \end{align*} So it is clear that $Z(\lambda)[M]$ is a homogeneous polynomial of degree $k$ in the variables $\langle \lambda + \rho, \check{\alpha_1} \rangle, \cdots, \langle \lambda + \rho, \check{\alpha_r} \rangle$. Thus by applying Lemma \ref{positive}, it follows that $Z(\lambda)[M] \geq 0$ when $\lambda$ lies inside the closure of the alcove $\underline{1}$. We will now prove the stronger statement that $Z(\lambda)[M] > 0$ when $\lambda$ lies inside $\underline{1}$. Suppose instead that $Z(\lambda)[M] = 0$.  From Proposition \ref{taylorharmonic}, and Proposition $6.3.25$ of \cite{cg}, we have that, for any $\mu$, $$ Z(\lambda)[M] = \frac{1}{|W|}\sum_{w \in W} Z(\lambda + w \mu)[M] $$ Pick $\mu$ sufficiently small so that $\lambda + w \mu$ lies in the alcove $\underline{1}$ for all $w \in W$; then $Z(\lambda + w \mu)[M] = 0$ for all $w \in W$. Since $\mu$ was arbitrary, this means that $Z(\lambda)[M] = 0$ on some neighbourhood of $\lambda$. Hence $Z(\lambda)[M] = 0$, contradicting the fact that $M$ was a non-zero object. This completes the proof of the first condition from Definition \ref{realvar}. 

Now we will check the second condition in Definition \ref{realvar}. In keeping with the notation used there, recall that $$\mathcal{C}_{\underline{w}, \underline{w'}}^n = \{ C \in \mathcal{C} \; | \; H^i_{\tau(\underline{w})}(C) \in  \mathcal{A}_{\underline{w}, \underline{w'}}^n \; \forall \; i \; \in \mathbb{Z} \}$$

First we need to show that the filtration $\{ 0 \} = \mathcal{C}_{\underline{w}, \underline{w'}}^2 \subset \mathcal{C}_{\underline{w}, \underline{w'}}^1 \subset \mathcal{C}$ is stable under the truncation functors for the t-structure $\tau(\underline{w'})$. Using Lemma \ref{easy}, we may reduce to the case where $w' = 1, w = s_{\alpha}$ (where $\alpha$ is a simple root). Thus we have: \begin{align*} H_{\tau(\underline{s_{\alpha}})}^n (C) &=  \Phi(\widetilde{s_{\alpha}}) H^n (\Phi(\widetilde{s_{\alpha}}^{-1}) C) \\ \mathcal{C}_{\underline{s_{\alpha}}, \underline{1}}^1 &= \{ C \in \mathcal{C} \; | \; H^i_{\tau(\underline{s_{\alpha}})}(C) \in  \mathcal{A}_{\underline{s_{\alpha}}, \underline{1}}^1 \} \\ &= \{ C \in \mathcal{C} \; | \; \Phi(\widetilde{s_{\alpha}}) H^n (\Phi(\widetilde{s_{\alpha}}^{-1}) C) \in \Phi(\widetilde{s_{\alpha}}) \mathcal{A}_{\underline{1}, \underline{s_{\alpha}}}^1 \} \\  &= \{ C \in \mathcal{C} \; | \; H^n (\Phi(\widetilde{s_{\alpha}}^{-1}) C) \in \mathcal{A}_{\underline{1}, \underline{s_{\alpha}}}^1 \} \end{align*}

We wish to show that if $C \in \mathcal{C}_{\underline{s_{\alpha}}, \underline{1}}^1$, then $\tau_{\leq i} C, \tau_{\geq i} C \in \mathcal{C}_{\underline{s_{\alpha}}, \underline{1}}^1$ for $i \in \mathbb{Z}$. 

Now $A = \Phi(\widetilde{s}_{\alpha}^{-1}) C$ satisfies the hypothesis of Lemma \ref{strong}, so: \begin{align*} \Phi(\widetilde{s}_{\alpha})(A) &= A[1] \Rightarrow \Phi(\widetilde{s}_{\alpha}^{-1}) C = C[-1] \\ & \therefore H^j(C) \in  \mathcal{A}^1_{\underline{1}, \underline{s}_{\alpha}} \forall i \in \mathbb{Z} \\  H^j(\tau_{\geq i} C) &= \begin{cases}  H^j(C) &\mbox{if } j \geq i \\ 0 &\mbox{if } j < i\end{cases} , H^j(\tau_{\leq i} C) = \begin{cases}  H^j(C) &\mbox{if } j \leq i \\ 0 &\mbox{if } j > i\end{cases}  \\ \Rightarrow H^j(\tau_{\geq i} C), H^j(\tau_{\leq i} C) &\in  \mathcal{A}^1_{\underline{1}, \underline{s}_{\alpha}} \forall i \in \mathbb{Z} \end{align*}

Applying Lemma \ref{strong} again, we get that: \begin{align*} \Phi(\tilde{s}_{\alpha}) (\tau_{\geq i} C) = \tau_{\geq i} C[1], \Phi(\tilde{s}_{\alpha}) (\tau_{\leq i} C) = \tau_{\leq i} C[1] \\ \Phi(\tilde{s}_{\alpha}^{-1}) (\tau_{\geq i} C) = \tau_{\geq i} C[-1], \Phi(\tilde{s}_{\alpha}^{-1}) (\tau_{\leq i} C) = \tau_{\leq i} C[-1] \\ \Rightarrow H^i( \Phi(\tilde{s}_{\alpha}^{-1}) (\tau_{\geq i} C)), H^i( \Phi(\tilde{s}_{\alpha}^{-1}) (\tau_{\leq i} C)) \in \mathcal{A}^1_{\underline{1}, \underline{s}_{\alpha}} \end{align*} It follows that if $C \in \mathcal{C}_{\underline{s_{\alpha}}, \underline{1}}^1$, then $\tau_{\leq i} C, \tau_{\geq i} C \in \mathcal{C}_{\underline{s_{\alpha}}, \underline{1}}^1$ for $i \in \mathbb{Z}$. 

We also need to show that the two t-structures on the quotient $\mathcal{C}_{\underline{s_{\alpha}}, \underline{1}}^n/\mathcal{C}_{\underline{s_{\alpha}}, \underline{1}}^{n+1}$ induced by $\tau(\underline{1})$ and $\tau(\underline{s_{\alpha}})$ differ by a shift of $[n]$, for $n = 0, 1$. For $n=0$, this follows from Lemma \ref{szero}; for $n=1$, this follows directly follow Lemma \ref{strong}. \end{proof}

\subsection{Proof of Theorem 2} \label{bridge}

We will apply the above result (which gives a construction of ``real variation of stability conditions'') to explicitly describe a sub-manifold in the space of Bridgeland stability conditions on $\mathcal{C} = D^b(\mathcal{O}_0^d)$. The arguments used in this section are identical to those  used in Section $4.2$ of \cite{abm}; all of the facts used there about exotic $t$-structures on $D^b(\text{Coh}_{\mathcal{B}_e}(U_e))$ have analogues in our set-up, so the same proof is applicable here. We describe the proof in some detail below; note, however, that the notation we use is different (for instance, $\mathfrak{h}^*_{\text{reg}}$ and $ \widetilde{V_{\text{reg}}}$ mean different things in their setup).

Recall that the Weyl group $W$ acts naturally on $\mathfrak{h}^* = \Lambda \otimes_{\mathbb{Z}} \mathbb{C}$; denote by $\mathfrak{h}^*_{\text{reg}}$ the union of the free orbits of $W$ on $\mathfrak{h}^*$. Let us fix a universal covering space  $\widetilde{\mathfrak{h}^*_{\text{reg}}}$ of $\mathfrak{h}^*_{\text{reg}}$. Also, let $\mathfrak{h}^*_{\mathbb{R}} = \Lambda \otimes_{\mathbb{Z}} \mathbb{R}$ be the real Cartan. Denote by $V^{\text{reg}}$ the following subspace of $\mathfrak{h}^*$ (here we identify $\mathfrak{h}^*_{\mathbb{R}} \times \mathfrak{h}^*_{\mathbb{R}} \simeq \mathfrak{h}^*$): $$ V^{\text{reg}} = \{ (\lambda, \mu) \in \mathfrak{h}^*_{\mathbb{R}} \times \mathfrak{h}^*_{\mathbb{R}} | (\lambda \in (\mathfrak{h}^*_{\mathbb{R}})^{\text{reg}}) \bigvee (\lambda \in \overline{A}, \mu \in A \text{ for some } A \in \textbf{Alc}) \} $$ 
Denote by $\widetilde{V_{\text{reg}}}$ the pre-image of $V_{\text{reg}}$ in $\widetilde{\mathfrak{h}^*_{\text{reg}}}$. Let us define the map $\textbf{Z}: \mathfrak{h}^* \rightarrow K^0(\mathcal{C})^*$ as follows; one should think of it as a complexification of the map $Z: \mathfrak{h}^*_{\mathbb{R}} \rightarrow K^0(\mathcal{C})_{\mathbb{R}}^*$. Again, identify $\mathfrak{h}^*_{\mathbb{R}} \times \mathfrak{h}^*_{\mathbb{R}} \simeq \mathfrak{h}^*$:
\begin{align*} \textbf{Z}(\lambda,\mu)[M] = Z(\lambda)[M] + \sqrt{-1} Z(\mu)[M] \end{align*}

\begin{theorem} There exists a unique map (of manifolds) $\imath$, from $\widetilde{V_{\text{reg}}}$ to the space of locally finite Bridgeland stability conditions on $\mathcal{C}$, such that: \begin{itemize} \item We have the commuting square $$\overline{\pi} \circ \imath= \sqrt{-1} \textbf{Z} \circ \pi$$ where $\overline{\pi}: \text{Stab}(\mathcal{C}) \rightarrow K^0(\mathcal{C})^*$, and $\pi: \widetilde{\mathfrak{h}}^*_{\text{reg}} \rightarrow \mathfrak{h}^*_{\text{reg}}$, are the natural projection maps. \item For any alcove $\underline{w}$, and any point $z \in \underline{w}$, the underlying t-structure of the stability condition $\imath(z)$ is $\tau(\underline{w})$. \item The map $\imath$ is compatible with the action of $\mathbb{B}_W$, which acts on the source by deck transformations, and on the target via the action on $\mathcal{C}$.  \end{itemize} \end{theorem}  \begin{proof} One readily verifies that the following set is a fundamental domain for the action of $W$ on $V^{\text{reg}}$ (here $\overline{\underline{1}}$ denotes the closure of the alcove $\underline{1}$): $$ S = \{ (\lambda, \mu) \in V^{\text{reg}} | (\lambda \in \underline{1}) \bigvee (\lambda \in \overline{\underline{1}}, \mu \in \underline{1}) \} $$ Thus a point in $\widetilde{V^{\text{reg}}}$ can be represented by a pair $(b,x)$, where $x \in S$ and $b \in \mathbb{B}_W$ represents a homotopy class of a path from $\underline{1}$ to some alcove $\underline{w}$; the projection to $V^{\text{reg}}$ is given by $(b, x) \rightarrow \overline{b}(x)$, where $\overline{b} \in W$ corresponds to $b \in \mathbb{B}_{aff}$. Now define the map $\imath$ as follows: $$ \imath(b, x) = \mathfrak{S}(\Phi(b)(\tau_{\underline{1}}), \sqrt{-1} \textbf{Z}(\overline{b}(x))) $$ Here $\mathfrak{S}$ refers to the map constructed by Bridgeland in Propositon $5.3$ of \cite{bridgeland}; there he shows that giving a stability condition is equivalent to giving a t-structure and a central charge map (in our case, $\Phi(b) \tau_{\underline{1}}$, and $\sqrt{-1} \textbf{Z}(\overline{b}(x))$ respectively). It is easy to see the map lands in the space of locally finite stability conditions, and that it is $\mathbb{B}_W$-equivariant. 

Let us now check continuity of this map. It suffices to check continuity at a point $(\lambda, \mu) \in S$ lying on the boundary of the region $\underline{1}$; here $\lambda \in F, \mu \in \underline{1}$, where $F$ is a face on the boundary of the two alcoves $\underline{1}$ and $\underline{s_{\alpha}}$ (for some simple reflection $s_{\alpha}$). In Theorem $1.2$ of \cite{bridgeland}, Bridgeland shows that there exists a neighbourhood of the point $z(\lambda, \mu) := \imath(1, (\lambda, \mu))$ in $\text{Stab}(\mathcal{C})$ which maps isomorphically to a neighbourhood of $\sqrt{-1} \textbf{Z}(\lambda, \mu)$ in $K^0(\mathcal{C})^*$. Thus to check continuity, it suffices to see that for a small neighbourhood $U$ of $\sqrt{-1} \textbf{Z}(\lambda, \mu) \in K^0(\mathcal{C})^*$, and a point $\tilde{z} \in \text{Im}(\imath) \subset \text{Stab}(\mathcal{C})$ mapping to $z \in \sqrt{-1} \textbf{Z}(s_{\alpha}(S)) \cap U$, the t-structure underlying $\tilde{z}$ is $\Phi(\tilde{s}_{\alpha}^{-1})(\tau_{\underline{1}})$. 

To prove this, it suffices to see that $\Phi(\tilde{s}_{\alpha}^{-1}) M$ lies in the heart of the t-structure underlying $\tilde{z}$, for any simple object $M \in \mathcal{O}_0^d$. We have two cases to consider: either $f_M$ has a zero of order $1$ on the $F$, or not. In the first case, we have that $\Phi(\tilde{s}_{\alpha}^{-1}) M = M[-1]$ using Proposition \ref{3.1}. Using Lemma 2 from \cite{abm}, since $M$ is stable with respect to the stability condition $z (\lambda, \mu)$, we can choose a sufficiently small neighbourhood $U$ such that $M$ is stable with respect to $\sigma'$ for any $\sigma' \in U$. Thus $M$ is stable with respect to $\tilde{z}$. Further, since $Z(\mu)[M] > 0, Z(\lambda)[M] = 0$, we have that: \begin{align*} \langle \overline{\pi} z(\lambda, \mu), [M] \rangle &= \sqrt{-1}(Z(\lambda)[M] + \sqrt{-1} Z(\mu)[M]) \in \mathbb{R}_{<0}  \end{align*} However, $\langle z, [M] \rangle$ lies in the lower half plane: indeed, when $\lambda' \in \underline{1}, (\lambda', \mu') \in S$, $\langle \imath(\tilde{s}_{\alpha}, (\lambda', \mu')), [M] \rangle$ lies in the lower half plane, since $Z(\lambda)[M] > 0$ and $[\tilde{s}_{\alpha}^{-1} (M)] = - [M]$. This now implies that the phase of $M$ lies in $(1,2)$, so that $M[-1]$ lies in the heart of the $t$-structure underlying $\tilde{z}$. 

Now suppose $f_M$ does not have a zero on $F$; then $\tilde{M} := \tilde{s}_{\alpha}^{-1} M$ lies in the heart of $\tau_{\underline{1}}$. Recall from Lemma \ref{szero} that we have a short exact sequence $0 \rightarrow M \rightarrow \tilde{M} \rightarrow M' \rightarrow 0$, with $d_{M'}$ having a zero on $F$, and $M$ being the only simple sub-module of $\tilde{M}$. Further, $\langle \overline{\pi} z(\lambda, \mu), [M] \rangle$ lies in the open upper half-plane, since $Z(\lambda)[M] > 0$ in this case using Proposition \ref{3.1}. Meanwhile, $\langle \overline{\pi} z(\lambda, \mu), [M'] \rangle \in \mathbb{R}_{<0}$ (from the argument used to prove the analogous statement in the last paragraph). For any non-zero submodule $N$ of $\tilde{M}$, $\tilde{M}/N$ is a submodule of $M'$. Hence, for some $s > 0$, $$\overline{\pi} z(\lambda, \mu) [N] =  \overline{\pi} z(\lambda, \mu) [\tilde{M}] + s $$ Thus any such $N$ has phase smaller than $\tilde{M}$; it follows that $\tilde{M}$ is stable and has phase in $(0, 1)$. \end{proof}

\section{Further directions}

\subsection{Leading coefficient polynomials} Recall that for Proposition \ref{Z}, we used the quantity $LC(M)$ instead of $\text{\underline{LC}}(M)$. However, the quantity $\text{\underline{LC}}(M)$ is slightly easier to define than $\text{LC}(M)$, and is in some sense more natural. We expect that Proposition \ref{Z} (and consequently, all the other results obtained in this paper) will continue to hold with $\text{\underline{LC}}(M)$ instead of $\text{LC}(M)$. This would be a consequence of Conjecture \ref{conj}; it would be interesting to understand this better. One approach is to use Beilinson-Bernstein localization theory, and to consider the singular support of the $D$-modules which correspond to these irreducible objects in category $\mathcal{O}$. Alternatively, it is conceivable that this conjecture would be amenable to a more elementary approach, perhaps by invoking general facts about Hilbert polynomials for graded polynomial rings.

\subsection{Generalization to parabolic and singular category $\mathcal{O}$} In this paper, we deal primarily with the regular block of category $\mathcal{O}$. We expect that these results are valid, more generally, for a (possibly singular) block of parabolic category $\mathcal{O}$; this should be a straightforward application of the techniques developed here. The special case of a maximal parabolic in $\mathfrak{sl}_n$, with two blocks of sizes $n-1$ and $1$ is of special interest, since in this case parabolic category $\mathcal{O}$ can be described explicitly as modules over a certain path algebra (which closely resembles the zig-zag algebra, and the pre-projective algebra in type $A$). We expect this example to be closely related to that studied by Bridgeland in \cite{bridgeland2}, involving stability conditions on resolutions of Kleinian singularities (which can be re-formulated as module categories over pre-projective algebras).

%It would be interesting to understand how Koszul duality between parabolic and singular blocks (a la Beilinson-Ginzburg-Soergel, \cite{bgs}) interacts with the leading coefficient polynomials $\underline{\text{LC}}(\lambda)$, and the action of the braid group $\mathbb{B}_W$. 

\subsection{A characteristic $p$ analogue} Here we sketch a conjectural characteristic $p$ analogue of the present work; it would be interesting to make this more precise, and check that these statements hold. Earlier we stated that the construction from \cite{abm} can be viewed, loosely, as a characteristic $p$ analogue of our construction; what we describe now is a more direct characteristic $p$ analogue of our construction, for $p$ sufficiently large. Consider the category $\mathcal{C}_{\lambda, e} := \text{Mod}^{\text{fg}}_{\lambda, e}(U \mathfrak{g})$ of modules over $U\mathfrak{g}$ with generalized central character $(\lambda, e)$, where $\lambda \in \mathfrak{h}^*$ specifies the action of the Harish-Chandra center, and the nilpotent $e \in \mathfrak{g}^*$ specifies the action of the Frobenius center.  Given a module $M \in \text{Mod}^{\text{fg}}_{\lambda, e}(U \mathfrak{g})$, let $\text{dim}(M)$ be its dimension. It is known that the category $\mathcal{C}_{\lambda, e}$ does not change as $p$ varies (for $p$ sufficiently large), so the module $M$ can be defined in characteristic $p$ for all primes $p$ sufficiently large; further, $\text{dim}(M)$ is a polynomial in $p$. Let us denote its leading coefficient by $\text{LC}(M)$, and its degree by $d(M)$; further, denote by $\mathcal{C}_{\lambda, e}^{\leq d}$ the sub-category consisting of all modules $M$ for which $d(M) \leq d$, and let $\mathcal{C}_{\lambda, e}^d = \mathcal{C}_{\lambda, e}^{\leq d} / \mathcal{C}_{\lambda, e}^{\leq d-1}$. Then the action of the affine braid group $\mathbb{B}_{aff}$ on $D^b(\mathcal{C}_{\lambda, e})$ induces an action on $D^b(\mathcal{C}_{\lambda, e}^d)$. Let us define the central charge map analogously: $$Z(\lambda)[M] = \text{LC}(T_{0 \rightarrow \lambda} M)$$ As in \cite{abm}, let $V = \mathfrak{h}^*$, and let $\Sigma$ consist of all affine co-root hyperplanes, so that the set of alcoves are in bijection with the affine Weyl group $W_{\text{aff}}$. For each $w \in W_{aff}$, define the t-structure $\tau(w)$ analogously (using the action of the automorphism corresponding to a lift $\tilde{w} \in \mathbb{B}_{aff}$ of $w$). We expect that this datum satisfies the conditions defining a real variation of stability conditions, and that one can use this to describe a sub-manifold in the space of Bridgeland stability conditions (following Section $4.2$ of \cite{abm}). The same method of proof employed here should be applicable for the most part; however, for Proposition \ref{Z}, one needs to find an alternative approach. Alternatively, one may attempt to deduce these results directly from the main result of \cite{abm}. Their result is, loosely speaking, stronger, since they doesn't use sub-quotients; however, it isn't immediately clear whether or not this follow as a consequence.

\subsection{Piecing together the action on the sub-quotients} Once the constructions in the last paragraph have been formalized, we expect that the construction in \cite{abm} can be obtained, in a very loose sense, by piecing them together. Here, we have constructed some examples of stability conditions (and real variations thereof) using Gelfand-Kirillov sub-quotients of category $\mathcal{O}$. Analogously to the characteristic $p$ case discussed in \cite{abm}, it would be interesting to piece these together, and construct an example of (real variations of) stability conditions using the entirety of category $\mathcal{O}$, instead of its sub-quotients. The main difficulty is that in positive characteristic, one can look at the dimension of a module, which is a finer invariant than the leading coefficient $\text{LC}(M)$ described above; however, for a module lying in category $\mathcal{O}$, it is not clear what one should use as a substitute for the dimension. 

\bibliographystyle{plain}

\end{document}